\documentclass[11pt,a4paper,reqno]{article}
\usepackage{amsmath,amssymb,amsthm,mathtools}
\usepackage[a4paper,margin=24mm]{geometry}
\usepackage{lmodern}
\usepackage{microtype}
\usepackage{booktabs,tabularx}
\usepackage{xcolor}
\usepackage{hyperref}

\definecolor{primary}{RGB}{18,53,91}
\hypersetup{
    colorlinks=true,
    linkcolor=primary,
    citecolor=primary,
    urlcolor=primary,
    pdftitle={Reconstruction of Arbitrary Torsion-Free Nilpotent (and Hypercentral) Groups from the Division Rings of Fractions of Their Rational Group Algebras},
    pdfauthor={Jinyu Lin, Xiaodong Wang}
}

\theoremstyle{plain}
\newtheorem{theorem}{Theorem}[section]
\newtheorem{lemma}[theorem]{Lemma}
\newtheorem{proposition}[theorem]{Proposition}
\newtheorem{corollary}[theorem]{Corollary}

\theoremstyle{definition}

\newtheorem{remark}[theorem]{Remark}

\numberwithin{equation}{section}

\newcommand{\Q}{\mathbb{Q}}
\newcommand{\Z}{\mathbb{Z}}
\newcommand{\N}{\mathbb{N}}
\newcommand{\Fp}{\mathbb{F}_p}
\newcommand{\Qbar}{\overline{\mathbb{Q}}}
\DeclareMathOperator{\Frac}{Frac}
\DeclareMathOperator{\FracOre}{Frac_{Ore}}
\DeclareMathOperator{\rank}{rank}
\DeclareMathOperator{\trdeg}{trdeg}
\DeclareMathOperator{\coker}{coker}
\DeclareMathOperator{\supp}{supp}
\DeclareMathOperator{\Tor}{Tor}
\DeclareMathOperator{\Hom}{Hom}
\DeclareMathOperator{\GL}{GL}
\DeclareMathOperator{\id}{id}

\title{\textbf{Reconstruction of Arbitrary Torsion-Free Nilpotent (and Hypercentral) Groups from the Division Rings of Fractions of Their Rational Group Algebras}}
\author{Jinyu Lin$^{1}$, Xiaodong Wang$^{2}$\\[0.6em]
\small $^{1}$\,Taishan College, Shandong University, Jinan, China\\[0.2em]
\small $^{2}$\,School of Mathematical Sciences, Ocean University of China, Qingdao, China\\[0.4em]
\small $^{1}$\,202400171064@mail.sdu.edu.cn, $^{2}$\,wxd8510@stu.ouc.edu.cn}
\date{}

\begin{document}
\maketitle

\begin{abstract}
Let $G$ and $H$ be torsion-free nilpotent groups of arbitrary cardinality, without any uniform upper bound imposed on their cardinality, torsion-free rank, number of generators, or nilpotency class. We prove that if the classical Ore division rings of fractions of their rational group algebras are isomorphic as rings, namely
\[
D_{\Q}(G) := \FracOre \Q[G] \cong \FracOre \Q[H] =: D_{\Q}(H),
\]
then $G$ and $H$ are isomorphic as groups ($G \cong H$). This result extends directly to $k$-algebra isomorphisms $D_k(G) \cong_k D_k(H)$ over any base field $k$ of characteristic zero, as well as to arbitrary torsion-free hypercentral groups (and, more generally, to torsion-free groups satisfying bi-orderability, the Ore condition, and nontrivial intersection of normal subgroups with the center). A key tool in our proof is a weight-preserving isomorphism theorem for rational function fields of free abelian groups: let $P$ and $Q$ be free abelian groups of arbitrary rank equipped with surjective homomorphisms onto a common abelian group $B$; if their rational function fields over a field $k$ containing an algebraically closed subfield of characteristic zero admit a $B$-weight-preserving field isomorphism $k(P) \cong_B k(Q)$, then there exists a weight-preserving group isomorphism $P \cong_B Q$. By combining the Malcev--Neumann power series division ring, lowest-degree coset extraction over central localizations, the lowest-term rational coefficient homomorphism, and an embedding into a torsion-free abelianized auxiliary function field, we reduce the division ring isomorphism to a weight-preserving rational function field isomorphism between the quotients $G/L$ and $H/L$, and explicitly construct a group isomorphism $G \cong H$ via a central modification.

\vspace{4pt}

\vspace{2pt}
\end{abstract}

\section{Introduction}

Determining the extent to which a group ring and its rings of fractions determine the isomorphism class of the underlying group is a classic question in algebra. Since Higman posed the isomorphism problem for integral group rings $\Z[G]$ in 1940, isomorphism problems across different coefficient rings and group classes have been studied extensively (see Passman \cite{Passman1977}). When the ring of integers $\Z$ is replaced by the rational field $\Q$ or a general field $k$ of characteristic zero, the group algebra $k[G]$ loses part of the arithmetic information carried by integral coefficients; for instance, the rational group algebra of a finite abelian group is determined solely by the number of its cyclic subgroups of each order, while for finite nilpotent groups there exist non-isomorphic groups with isomorphic rational (and even integral) group rings.

When one turns to torsion-free nilpotent groups, the division rings of fractions of their group algebras display a different behavior. By a classical theorem of Malcev \cite{Malcev1949,Malcev1951}, every torsion-free nilpotent group $G$ admits a bi-invariant total ordering $<$, so its group algebra $k[G]$ has no zero divisors. Because nilpotent groups are amenable, a theorem of Tamari \cite{Tamari1954} (see also Bartholdi and Kielak \cite{BartholdiKielak2019} for a modern proof) guarantees that $k[G]$ satisfies the left and right Ore conditions and thus possesses a classical Ore division ring of fractions $D_k(G) := \FracOre k[G]$. At the same time, by the work of Neumann \cite{Neumann1949}, formal power series with well-ordered support on a bi-ordered group $(G, <)$ form the Malcev--Neumann division ring $k((G, <))$, into which $D_k(G)$ embeds naturally by the universal property of Ore fractions (see Lichtman \cite{Lichtman1978,Lichtman2000} and Cohn \cite{Cohn1995}). In the finitely generated setting, Farkas, Schofield, Snider, and Stafford \cite{Farkas1982} proved in 1982 that finitely generated torsion-free nilpotent groups are uniquely determined by the division rings of fractions of their group algebras, while constructing a counterexample showing that the analogous statement fails for finitely generated torsion-free supersolvable groups.

In this paper, we study the case of general torsion-free nilpotent groups: let $G$ and $H$ be arbitrary torsion-free nilpotent groups (without any uniform upper bound on cardinality, rank, number of generators, or nilpotency class); if their rational group division rings of fractions $D_{\Q}(G)$ and $D_{\Q}(H)$ are isomorphic, does it necessarily follow that $G \cong H$?

We give an affirmative answer to this question. Throughout the paper, we use multiplicative notation in the group-theoretic sections and additive notation in the sections on free abelian groups and weights; for a torsion-free abelian group $A$, we write $k(A) := \Frac k[A]$. By the standard definition, a nilpotent group is a group whose lower central series terminates in finitely many steps. Our main theorem is as follows:

\begin{theorem}\label{thm:main}
Let $G$ and $H$ be arbitrary torsion-free nilpotent groups. If there exists a division ring isomorphism
\[
D_{\Q}(G) \cong D_{\Q}(H),
\]
then $G \cong H$.
\end{theorem}

The division ring isomorphism in Theorem \ref{thm:main} need not be assumed a priori to fix $\Q$, since any unital ring isomorphism between division rings of characteristic zero automatically fixes the prime field $\Q$ pointwise. Moreover, because the proof uses no special arithmetic properties of $\Q$ beyond having characteristic zero and admitting an algebraic closure, and because the arguments in Sections 7--11 depend only on the bi-orderability of the groups and their central quotients, the Ore condition, and the property that nontrivial normal subgroups intersect the center nontrivially, Theorem \ref{thm:main} extends directly in Section 12 to $k$-algebra isomorphisms over any base field $k$ of characteristic zero (Corollary \ref{cor:12.1}), as well as to arbitrary torsion-free hypercentral groups (torsion-free groups whose transfinite upper central series exhausts the whole group) and general torsion-free groups satisfying the corresponding structural conditions (Theorem \ref{thm:12.5}). Combined with the counterexample of Farkas et al.\ \cite{Farkas1982} for torsion-free supersolvable groups, this shows that within the hierarchy of solvable groups, the property that the division ring of fractions determines the isomorphism class of the group reaches its optimal boundary at the level of torsion-free nilpotent and torsion-free hypercentral groups.

To prove Theorem \ref{thm:main}, we first study in Sections 2--6 the weight-preserving isomorphism problem for rational function fields of free abelian groups under character actions. Let $k$ be a field containing an algebraically closed subfield $k_*$ of characteristic zero, let $B$ be an abelian group, and let $P$ and $Q$ be free abelian groups equipped with surjective homomorphisms $\rho_P : P \twoheadrightarrow B$ and $\rho_Q : Q \twoheadrightarrow B$. Each character $\chi \in \Hom(B, k_*^\times)$ induces a field automorphism $\sigma_\chi^P$ on $k(P)$ (and $\sigma_\chi^Q$ on $k(Q)$) via $X^a \mapsto \chi(\rho_P(a))X^a$. If a field isomorphism $\Phi : k(P) \xrightarrow{\sim} k(Q)$ satisfies $\Phi \circ \sigma_\chi^P = \sigma_\chi^Q \circ \Phi$ for all characters $\chi$, we call $\Phi$ a $B$-weighted isomorphism and write $k(P) \cong_B k(Q)$. In Section 6, we prove the following theorem:

\begin{theorem}[see Theorem \ref{thm:6.1}]\label{thm:weighted-intro}
Let $k$ be a field containing an algebraically closed subfield $k_*$ of characteristic zero, let $B$ be an arbitrary abelian group, and let $P$ and $Q$ be free abelian groups of arbitrary rank with surjective homomorphisms $\rho_P : P \twoheadrightarrow B$ and $\rho_Q : Q \twoheadrightarrow B$. If there exists a $B$-weighted $k$-field isomorphism $\Phi : k(P) \xrightarrow{\sim} k(Q)$, then there exists a weight-preserving group isomorphism
\[
\beta : P \xrightarrow{\sim} Q, \qquad \rho_Q \beta = \rho_P.
\]
\end{theorem}

The proof of Theorem \ref{thm:weighted-intro} is divided into three parts. Section 2 uses the divisibility of $k_*^\times$ to show that characters separate subgroups and elements, defines the lowest-exponent valuation $\nu : k(Q)^\times \to Q$ induced by a translation-invariant total ordering on $Q$, and constructs from $\Phi$ a weight-preserving group homomorphism $f : P \to Q$, $a \mapsto \nu(\Phi(X^a))$. Sections 3 and 4 treat the case where there exists a prime $p$ for which $T = \ker \rho_P$ satisfies $\delta_p(P, T) = \dim_{\Fp}(T+pP)/pP < \infty$: using transcendence degrees, the long exact sequence for $-\otimes_{\Z} \Z/p\Z$, and the Snake Lemma, we prove $\rank(\ker f) = \rank(Q/f(P)) < \infty$; finite field extension degrees then show that the quotient $Q/f(P)$ is finitely generated, and an exterior-power criterion for generating sets of finitely generated abelian groups due to Reichstein and Youssin \cite[Lemma 2.4, Corollary 2.5]{ReichsteinYoussin2000} yields a weight-preserving isomorphism. Sections 5 and 6 handle the countable-rank case where $\delta_p(P, T) = \delta_p(Q, S) = \aleph_0$ for all primes $p$ and the general arbitrary-rank case: in countable rank, the Smith normal form is used to recursively select primitive vectors with pairwise disjoint supports to split off direct summands isomorphic to $\Z^{(\N)}$, and the fiber product $W = P \times_B Q$ completes the proof; in arbitrary rank, transfinite induction and block upper-triangular matrix modifications extend the countable-rank result to the entire group.

In the second half of the paper (Sections 7--12), we return to torsion-free nilpotent groups $G$ and $H$. Sections 7--9 establish a series of structural results for $G$ and $H$, including the characterization of the center $Z(D_{\Q}(G)) = \Frac \Q[Z(G)]$ (Proposition \ref{prop:7.2}) and the quotient embedding $H/Z(H) \hookrightarrow G/Z(G)$, replace $H$ by an isomorphic subgroup of $F^\times G$ via the lowest-coset-term map on the central localization (Proposition \ref{prop:8.10}), and eliminate $\Q^\times$-scalars via the lowest-term coefficient homomorphism to obtain a common group $E = KG = KH$ with $K = Z(E)$. Combining this with the fact that the quotient group $\Delta_k(A) = k(A)^\times / (k^\times X^A)$ of any torsion-free abelian group $A$ is free abelian (proved in our earlier preprint \cite{LinWang2026}; for self-containment, a standalone proof is included in Section 9), we prove that there exist free abelian subgroups $C_G, C_H \le K$ such that $E = G \times C_G = H \times C_H$. Section 10 sets $L = G \cap H$, quotients out the isolator $I = \sqrt[E]{E'} \le L$ of the derived subgroup to obtain the torsion-free abelian group $A = E/I$, extends the constant field to $\Qbar$, and embeds both sides into a common commutative auxiliary field $\Omega = \Frac \widehat{F}[A]$ via the injective homomorphism $\Theta(a, k) = kX^{a - \bar{k}}$; we prove that $\Omega = k_{\mathrm{aux}}(G/L) = k_{\mathrm{aux}}(H/L)$ gives a weighted field isomorphism over $B = E/(KL)$, whence Theorem \ref{thm:weighted-intro} yields a weight-preserving isomorphism $\beta : G/L \xrightarrow{\sim} H/L$. Finally, Section 11 lifts the difference map $d(p) = \beta(p) - p \in KL/L$ to a homomorphism $\eta : G/L \to K$ and proves that $\Psi(g) = g\eta(gL)$ gives a group isomorphism $G \xrightarrow{\sim} H$; Section 12 extends the above results to arbitrary base fields of characteristic zero, torsion-free hypercentral groups, and general torsion-free groups satisfying the corresponding structural conditions; and Section 13 summarizes the paper and discusses several directions for future research.

\section{Characters and the Lowest-Exponent Valuation}

In this section, we establish the basic tools for weighted isomorphisms between rational function fields $k(P)$ and $k(Q)$ of free abelian groups. Using the divisibility of the multiplicative group $k_*^\times$ of an algebraically closed subfield of characteristic zero, we first show that characters separate subgroups of $B$ from elements outside them. We then fix a translation-invariant total ordering on $Q$ to define the lowest-exponent valuation $\nu : k(Q)^\times \to Q$, and use it to construct a weight-preserving group homomorphism $f : P \to Q$ from a weighted field isomorphism $\Phi : k(P) \xrightarrow{\sim} k(Q)$.

(Conventions.) Let $P$ and $Q$ be free abelian groups, let $\rho_P : P \twoheadrightarrow B$ and $\rho_Q : Q \twoheadrightarrow B$ be surjective homomorphisms, and let $B$ be an abelian group. The group operations on $P$ and $Q$ are written additively, and monomials in the group algebras $k[P]$ and $k[Q]$ are denoted by $X^a$ ($a \in P$) and $Y^q$ ($q \in Q$).

For the abelian group $B$, $\chi : B \to k_*^\times$ denotes a character, where $k$ is a field of $\operatorname{char} 0$ containing an algebraically closed field $k_* \subseteq k$ as a subfield.

\begin{lemma}\label{lem:2.1}
If $B_0 \le B$ is a subgroup and $b \notin B_0$, then there exists a character $\chi : B \to k_*^\times$ such that
\[
\chi|_{B_0} = 1 \quad \text{and} \quad \chi(b) \neq 1.
\]
\end{lemma}

\begin{proof}
Pass to $B/B_0$, in which $\bar{b} = b + B_0$ is nonzero. The element $\bar{b}$ generates a cyclic subgroup $\langle \bar{b} \rangle$. We define a character on $\langle \bar{b} \rangle$ (not equal to $1$ at $\bar{b}$): if the order of $\bar{b}$ is $n > 1$, choose a primitive $n$-th root of unity $\zeta \in k_*$ and set $\bar{b} \longmapsto \zeta$; if the order of $\bar{b}$ is infinite, choose any $\lambda \in k_*^\times \setminus \{1\}$ and set $\bar{b} \longmapsto \lambda$. Since $k_*^\times$ is divisible, $k_*^\times$ is an injective object in the category of abelian groups, so this character extends to $B/B_0$; composing with $B \to B/B_0$ gives a character
\[
\chi : B \longrightarrow B/B_0 \longrightarrow k_*^\times.
\]
For every $b_0 \in B_0$, its image in $B/B_0$ is trivial, so $\chi(b_0) = 1$, whereas $\chi(b) = \zeta$ (or $\lambda$) $\neq 1$.
\end{proof}

For each character $\chi : B \to k_*^\times$, define a field automorphism on $k(P)$ (where $P$ is a free abelian group) by
\[
\sigma_\chi^P(X^a) = \chi(\rho_P(a))X^a, \qquad \sigma_\chi^P|_k = \id.
\]
We define $\sigma_\chi^Q$ similarly. Then, for an isomorphism $\Phi : k(P) \xrightarrow{\sim} k(Q)$, if for every $\chi$ we have a commutative diagram
\[
\begin{array}{ccc}
k(P) & \xrightarrow{\sigma_\chi^P} & k(P) \\[6pt]
\Phi \Big\downarrow & & \Big\downarrow \Phi \\[6pt]
k(Q) & \xrightarrow[\sigma_\chi^Q]{} & k(Q)
\end{array}
\]
we say that $\Phi$ is a $B$-weighted isomorphism.

Since $Q$ is free abelian, we may choose a translation-invariant total ordering $<$ on $Q$. For a nonzero element
\[
u = \sum_{q \in F} c_q Y^q \in k[Q], \qquad F \subseteq Q \text{ a finite set},
\]
we define $\nu(u) = \min\{q : c_q \neq 0\}$; in other words, $\nu(u)$ is the group element in $Q$ corresponding to the nonzero term of $u$ with the smallest exponent.

\begin{lemma}\label{lem:2.2}
For nonzero $u, v \in k[Q]$, we have
\[
\nu(uv) = \nu(u) + \nu(v).
\]
Hence $\nu$ extends uniquely to a group homomorphism
\[
\nu : k(Q)^\times \longrightarrow Q,
\]
given by $\nu(u/v) = \nu(u) - \nu(v)$, satisfying $\nu(c) = 0$ ($c \in k^\times$) and $\nu(Y^q) = q$.
\end{lemma}

\begin{proof}
Let $q_0 = \nu(u)$ and $r_0 = \nu(v)$. Because the total ordering on $Q$ is translation-invariant, the term with the smallest exponent in the product must come from the product of $Y^{q_0}$ and $Y^{r_0}$: since $q_1 > q_2 \Rightarrow q_1 + r > q_2 + r$, any exponent $q > q_0$ appearing in $u$ satisfies $q + r > q_0 + r_0$, and similarly for any exponent $r > r_0$ appearing in $v$; hence the lowest term is $Y^{q_0}Y^{r_0} = Y^{q_0 + r_0}$.

Because the coefficients of $Y^{q_0}$ and $Y^{r_0}$ are nonzero, the coefficient of $Y^{q_0 + r_0}$ is the product of their coefficients and is also nonzero. Thus the lowest term does not vanish:
\[
\therefore \quad \nu(uv) = q_0 + r_0.
\]
The extension of $\nu$ to $k(Q)^\times$ is immediate, defined by $\nu(u/v) = \nu(u) - \nu(v)$.
\end{proof}

If $u, v \neq 0$ and $\nu(u) < \nu(v)$, then $\nu(u + v) = \nu(u)$, because $v$ has no effect on the lowest-degree term of $u$.

\begin{lemma}\label{lem:2.3}
If $0 \neq z \in k(Q)$ and there exists $b \in B$ such that for every character $\chi$ we have
\[
\sigma_\chi^Q(z) = \chi(b)z,
\]
then $\rho_Q(\nu(z)) = b$.
\end{lemma}

\begin{proof}
First suppose $z = u \in k[Q]$, with lowest-degree term $cY^{q_0}$. For $u = \sum_{i=1}^n c_i Y^{q_i}$,
\[
\sigma_\chi^Q(u) = \sum_{i=1}^n c_i \chi(\rho_Q(q_i))Y^{q_i}.
\]
Thus $\chi(\rho_Q(q_0)) = \chi(b) \Rightarrow \chi(\rho_Q(q_0) - b) = 1$.

By Lemma \ref{lem:2.1} applied with $B_0 = 0$, we obtain $\rho_Q(q_0) = b \Rightarrow \rho_Q(\nu(z)) = b$.

If $z = u/v \in k(Q)$, then $\nu(z) = \nu(u) - \nu(v)$, so
\[
\chi(\rho_Q(\nu(z))) = \chi\bigl(\rho_Q(\nu(u)) - \rho_Q(\nu(v))\bigr).
\]
Here $\sigma_\chi^Q(z) = \chi(b)z$ means that
\[
\frac{\chi(\rho_Q(\nu(u)))}{\chi(\rho_Q(\nu(v)))} = \chi\bigl(\rho_Q(\nu(u)) - \rho_Q(\nu(v))\bigr) = \chi(b),
\]
which implies $\chi(\rho_Q(\nu(z)) - b) = 1 \Rightarrow \rho_Q(\nu(z)) = b$ (again taking $B_0 = 0$).
\end{proof}

Define
\[
\begin{aligned}
f : P &\longrightarrow Q, \\
a &\longmapsto \nu\bigl(\Phi(X^a)\bigr).
\end{aligned}
\]

\begin{lemma}\label{lem:2.4}
$f$ is a homomorphism, and $\rho_Q \circ f = \rho_P$.
\end{lemma}

\begin{proof}
We have
\[
\begin{aligned}
f(a + a') &= \nu\bigl(\Phi(X^{a+a'})\bigr) \\
&= \nu\bigl(\Phi(X^a)\Phi(X^{a'})\bigr) \\
&= f(a) + f(a').
\end{aligned}
\]
On the other hand, $\sigma_\chi^P(X^a) = \chi(\rho_P(a))X^a$. Since $\Phi \circ \sigma_\chi^P = \sigma_\chi^Q \circ \Phi$, we have
\[
\sigma_\chi^Q\bigl(\Phi(X^a)\bigr) = \Phi\bigl(\sigma_\chi^P(X^a)\bigr) = \chi(\rho_P(a))\Phi(X^a).
\]
Thus $X^a$ and $\Phi(X^a)$ have the same weight, and by Lemma \ref{lem:2.3}, $\rho_Q\bigl(\nu(\Phi(X^a))\bigr) = \rho_P(a)$ because $\Phi(X^a)$ satisfies $\sigma_\chi^Q(\Phi(X^a)) = \chi(\rho_P(a))\Phi(X^a)$.
\end{proof}

Accordingly, if we define $N = \ker f$ and $T = \ker \rho_P$, then
\[
N = \ker f \subseteq \ker \rho_P = T.
\]
Moreover, $\rho_Q(f(P)) = B$: setting $M = f(P)$, since $\rho_P = \rho_Q \circ f$ is surjective, for every $b \in B$ there exists $p \in P$ such that $\rho_P(p) = \rho_Q(f(p)) = b$; restricting the domain of $\rho_Q$ to $f(P)$ therefore still maps onto $B$, so $\rho_Q(f(P)) = B$.

Clearly $M$ is free abelian, being a subgroup of the free abelian group $Q$.

\begin{lemma}\label{lem:2.5}
There exists a free abelian subgroup $R \le P$ such that
\[
P = N \oplus R, \qquad f|_R : R \xrightarrow{\sim} M;
\]
that is, $R$ is isomorphic to $M$.
\end{lemma}

\begin{proof}
The homomorphism $f : P \to M$ gives a short exact sequence
\[
0 \longrightarrow N \longrightarrow P \xrightarrow{f} M \longrightarrow 0.
\]
Clearly $N \to P$ is injective, $P \to M$ is surjective, and the image of $N \to P$ is $N = \ker f$, so the sequence is exact.

Since $M$ is a free abelian group, it is a projective $\Z$-module, and hence this short exact sequence splits. Choosing a section $R$ of $M$ yields
\[
P = N \oplus R.
\]
\end{proof}

\begin{lemma}\label{lem:2.6}
If $N = 0$, then $f : P \xrightarrow{\sim} Q$, and $f$ preserves $B$-weights.
\end{lemma}

\begin{proof}
$N = 0 \Rightarrow f$ is injective. Take any nonzero element
\[
u = \sum_{i=1}^n c_i X^{a_i} \neq 0 \in k[P],
\]
where the $a_i$ are pairwise distinct and $c_i \neq 0$. Since $f$ is injective, the elements $f(a_i)$ are pairwise distinct. Consequently, the lowest-exponent terms of $\Phi(c_i X^{a_i})$ are pairwise distinct (as $\nu(\Phi(c_i X^{a_i})) = f(a_i)$). Thus the lowest-exponent term in $\Phi(u)$ cannot be canceled by other terms, so
\[
\nu(\Phi(u)) = \nu\bigl(\Phi(X^{a_j})\bigr) = f(a_j) \in f(P),
\]
where $X^{a_j}$ is the term minimizing $f(a_i)$.

For any fraction $u/v \in k(P)^\times$, we likewise have
\[
\nu\bigl(\Phi(u/v)\bigr) = \nu(\Phi(u)) - \nu(\Phi(v)) \in f(P).
\]
Now take any $q \in Q$. By the surjectivity of $\Phi$, there exists $z \in k(P)^\times$ such that $\Phi(z) = Y^q$, whence
\[
q = \nu(Y^q) = \nu(\Phi(z)) \in f(P).
\]
Thus $f$ is surjective, and $f$ preserves weights because $\rho_Q \circ f = \rho_P$.
\end{proof}

Throughout the remainder of our analysis of $f$, we may therefore assume that $N = \ker f \neq 0$.

\section{Finite Modulo-\texorpdfstring{$p$}{p} Dimension and Rank Computations}

Starting in this section, we analyze the kernel $N = \ker f$ and cokernel $D = Q/f(P)$ of the weight-preserving homomorphism $f : P \to Q$. We prove that if, for some prime $p$, the image of $T = \ker \rho_P$ in $P/pP$ has finite dimension $\delta_p(P, T)$, then relative transcendence degrees, the long exact sequence for reduction modulo $p$, and the Snake Lemma imply $\delta_p(P, T) = \delta_p(Q, S)$ and $\rank N = \rank(Q/f(P)) < \infty$.

For a prime $p$, define
\[
\delta_p(P, T) := \dim_{\Fp} \frac{T + pP}{pP}.
\]
Here $(T + pP)/pP$ is the image of $T$ under the projection $P \longrightarrow P/pP$, which in general is not equal to $T/pT$.

\begin{proposition}\label{prop:3.1}
If there exists a prime $p$ such that $\delta_p(P, T) < \infty$, then
\[
\rank N = \rank(Q/f(P)) < \infty.
\]
\end{proposition}

The remainder of this entire section is devoted to proving Proposition \ref{prop:3.1}.

Let $d := \delta_p(P, T)$. By Lemma \ref{lem:2.5}, $P = N \oplus R$ and $N \subseteq T$. Consider the natural map
\[
N/pN \longrightarrow (T + pP)/pP.
\]
If $n \in N$ satisfies $n + pP = 0$, then $n \in pP$, so
\[
\begin{aligned}
n &\in N \cap pP \\
&= N \cap (pN \oplus pR) \\
&= pN \qquad (\text{since } N \cap R = \{0\}).
\end{aligned}
\]
Thus $N/pN \longrightarrow (T + pP)/pP$ is injective, and hence
\[
m := \rank N \le d.
\]
If $m = 0$, then $N$ is a free abelian group of rank $0$, so $N = 0$, $P \cong Q$, and $Q/f(P) = 0$. Henceforth we assume $m > 0$.

Set $F = k(Q)$ and $F_R = \Phi(k(R))$ (note that $R \cong M = f(P)$). Then we have
\[
\nu(F_R^\times) = f(P).
\]
Indeed, if $0 \neq u = \sum_i c_i X^{\gamma_i} \in k[R]$ with $\gamma_i$ pairwise distinct, then $f|_R$ is injective (since $\ker(f|_R) = N \cap R = 0$). Consequently, the minimal exponents $f(\gamma_i)$ of the terms $\Phi(c_i X^{\gamma_i})$ are pairwise distinct, so the minimal exponent of $\Phi(u)$ is simply the smallest $f(\gamma_i)$. Thus $\nu(\Phi(u)) \in f(R) \subseteq f(P)$, and taking differences for fractions $u/v$ gives $\nu(F_R^\times) \subseteq f(P)$. Conversely, since $P = N \oplus R$ and $f(N) = 0$, we have $f(R) = f(P)$, and every $m \in f(P)$ can be written as $m = f(r) = \nu(\Phi(X^r))$ for some $r \in R$. Hence $\nu(F_R^\times) = f(P)$.

Since $P = N \oplus R$, choose a basis $n_1, \dots, n_m$ of $N$; then
\[
k(P) = k(R)(X^{n_1}, \dots, X^{n_m}).
\]
Applying $\Phi$ and setting $t_i = \Phi(X^{n_i})$, we have
\[
\Phi(k(P)) = \Phi\bigl(k(R)(X^{n_1}, \dots, X^{n_m})\bigr),
\]
\[
F = k(Q) = F_R(t_1, \dots, t_m).
\]
The elements $t_1, \dots, t_m$ are algebraically independent over $F_R$, so this is a purely transcendental extension and
\[
\trdeg_{F_R} F = m.
\]
(Indeed, in $P = N \oplus R$, every $a \in P$ decomposes uniquely as
\[
a = r + e_1 n_1 + \dots + e_m n_m, \qquad r \in R, \ e_i \in \Z,
\]
so $k[P] = k[R][X^{\pm n_1}, \dots, X^{\pm n_m}]$. Suppose there exists a nonzero Laurent polynomial relation
\[
F(Z_1, \dots, Z_m) = \sum_e c_e Z_1^{e_1} \cdots Z_m^{e_m} = 0,
\]
where $e = (e_1, \dots, e_m) \in \Z^m$ and $c_e \in k(R)$, with only finitely many $c_e \neq 0$. Substituting $Z_i := X^{n_i}$ gives
\[
\sum_e c_e X^{e_1 n_1 + \dots + e_m n_m} = 0.
\]
Clearing denominators by a common nonzero multiplier $d \in k[R] \setminus \{0\}$ such that $a_e = dc_e \in k[R]$, we obtain
\[
\sum_e a_e X^{e_1 n_1 + \dots + e_m n_m} = 0, \qquad a_e \in k[R].
\]
Writing each $a_e \in k[R]$ as
\[
a_e = \sum_{r \in E_e} b_{e,r} X^r, \qquad b_{e,r} \in k, \ E_e \subseteq R \text{ finite},
\]
we have
\[
\sum_e \sum_{r \in E_e} b_{e,r} X^{r + e_1 n_1 + \dots + e_m n_m} = 0.
\]
By the direct sum $P = N \oplus R$, if
\[
r + \sum_{i=1}^m e_i n_i = r' + \sum_{i=1}^m e_i' n_i,
\]
then
\[
r - r' = \sum_{i=1}^m (e_i' - e_i) n_i \in R \cap N = \{0\},
\]
so $r = r'$ and $\sum_{i=1}^m (e_i' - e_i) n_i = 0$; since $n_1, \dots, n_m$ form a basis of $N$, we have $r = r'$ and $e_i = e_i'$ for all $i$. Since the monomials $X^p$ ($p \in P$) form a $k$-basis of $k[P]$, distinct monomials are linearly independent over $k$. Thus all $b_{e,r} = 0 \Rightarrow a_e = 0 \Rightarrow c_e = 0$, a contradiction. Hence $X^{n_1}, \dots, X^{n_m}$ are algebraically independent over $k(R)$. Since $\Phi$ is a field isomorphism with $F_R = \Phi(k(R))$ and $t_i = \Phi(X^{n_i})$, if $t_1, \dots, t_m$ were algebraically dependent over $F_R$, there would exist a nonzero polynomial $G(Z_1, \dots, Z_m) \in F_R[Z_1, \dots, Z_m]$ with $G(t_1, \dots, t_m) = 0$; applying $\Phi^{-1}$ would yield a nontrivial polynomial relation among $X^{n_1}, \dots, X^{n_m}$ over $k(R)$, a contradiction! Therefore $\trdeg_{F_R} F = m$.)

\begin{lemma}\label{lem:3.2}
$\rank(Q/f(P)) \le m$.
\end{lemma}

\begin{proof}
Set $r = \rank(Q/f(P))$. Take any finitely many linearly independent elements $\bar{q}_1, \dots, \bar{q}_s$ in $Q/f(P)$, and choose representatives $q_1, \dots, q_s \in Q$. Then $Y^{q_1}, \dots, Y^{q_s}$ are algebraically independent over $F_R$: if there were a nonzero polynomial relation
\[
\sum_e a_e Y^{e_1 q_1 + \dots + e_s q_s} = 0, \qquad a_e \in F_R = \Phi(k(R)),
\]
we could apply the valuation $\nu$. Since $\nu(F_R^\times) = f(P)$, the valuation of each nonzero term lies in the coset
\[
f(P) + e_1 q_1 + \dots + e_s q_s,
\]
and distinct multi-indices $e$ correspond to distinct cosets—otherwise
\[
f(P) + e_1 q_1 + \dots + e_s q_s = f(P) + e_1' q_1 + \dots + e_s' q_s
\]
would imply $e_i = e_i'$ for all $i$ by the linear independence of $\bar{q}_1, \dots, \bar{q}_s$ in $Q/f(P)$, a contradiction. Because all nonzero terms have valuations in distinct cosets (and hence distinct valuations in $Q$), the term of smallest valuation cannot be canceled, a contradiction!

Thus $Y^{q_1}, \dots, Y^{q_s}$ are algebraically independent over $F_R$, so $s \le \trdeg_{F_R} F = m$. Since every finite linearly independent subset of $Q/f(P)$ has at most $m$ elements, we conclude that
\[
\rank(Q/f(P)) \le m.
\]
\end{proof}

Set $D = Q/f(P)$, and define
\[
a = \dim_{\Fp} D[p], \qquad b = \dim_{\Fp} D/pD,
\]
where $D[p] = \{d \in D : pd = 0\}$ is the $p$-torsion subgroup of $D$.

Tensoring the short exact sequence
\[
0 \longrightarrow f(P) \longrightarrow Q \longrightarrow D \longrightarrow 0
\]
with $\Z/p\Z$ yields the long exact sequence (recalling $f(P) = M$)
\[
\begin{aligned}
\Tor_1^{\Z}(M, \Z/p\Z) &\longrightarrow \Tor_1^{\Z}(Q, \Z/p\Z) \longrightarrow \Tor_1^{\Z}(D, \Z/p\Z) \longrightarrow \\
M/pM &\longrightarrow Q/pQ \longrightarrow D/pD \longrightarrow 0.
\end{aligned}
\]
From the resolution $0 \to \Z \xrightarrow{\times p} \Z \to \Z/p\Z \to 0$, we have $\Tor_1^{\Z}(D, \Z/p\Z) \cong D[p]$, and similarly $\Tor_1^{\Z}(M, \Z/p\Z) = M[p]$ and $\Tor_1^{\Z}(Q, \Z/p\Z) = Q[p]$. Since both $Q$ and $M = f(P)$ are free abelian, $Q[p] = M[p] = 0$, and the exact sequence above simplifies to
\[
0 \longrightarrow D[p] \longrightarrow M/pM \longrightarrow Q/pQ \longrightarrow D/pD \longrightarrow 0.
\]

Since $P = N \oplus R$ and $f|_R : R \to M$ is an isomorphism, $f$ induces modulo $p$ the map
\[
\bar{f} : P/pP \longrightarrow Q/pQ,
\]
which satisfies
\[
\dim_{\Fp} \ker \bar{f} = m + a, \qquad \dim_{\Fp} \coker \bar{f} = b.
\]
Indeed,
\[
P/pP = (N/pN) \oplus (R/pR), \qquad R/pR \cong M/pM.
\]
Since $N = \ker f$, $N/pN$ lies in $\ker \bar{f}$ with $\dim_{\Fp} N/pN = m$, while the kernel of $M/pM \to Q/pQ$ is precisely $D[p]$ of dimension $a$; hence $\ker \bar{f}$ is the direct sum of $N/pN$ and the lift of $\ker(M/pM \to Q/pQ)$ in $R/pR$, giving $\dim_{\Fp} \ker \bar{f} = m + a$. Moreover, the image of $\bar{f}$ coincides with that of $M/pM \to Q/pQ$ because $f(P) = M$ and $R \to M$ is surjective, so the cokernels agree:
\[
\coker \bar{f} = D/pD, \qquad \dim_{\Fp} \coker \bar{f} = b.
\]

Furthermore, since $\rho_P$ and $\rho_Q$ are surjective onto $B$, we have a short exact sequence (with $T = \ker \rho_P$)
\[
0 \longrightarrow \frac{T + pP}{pP} \longrightarrow P/pP \longrightarrow B/pB \longrightarrow 0.
\]
Proof: if $\rho_P(a) \in pB$, write $\rho_P(a) = pb$; by surjectivity there exists $x \in P$ with $\rho_P(x) = b$, so $\rho_P(a - px) = pb - pb = 0 \Rightarrow a - px \in T$, whence $a \in T + pP$. Conversely, if $a = t + px$ with $t \in T$, then $\rho_P(a) = \rho_P(t) + p\rho_P(x) \in pB$, so
\[
\frac{T + pP}{pP} = \ker(P/pP \to B/pB),
\]
establishing exactness. Similarly, for $S = \ker \rho_Q$, we have
\[
0 \longrightarrow \frac{S + pQ}{pQ} \longrightarrow Q/pQ \longrightarrow B/pB \longrightarrow 0.
\]
Since $\rho_Q f = \rho_P$, we have the commutative diagram with exact rows:
\[
\begin{array}{ccccccccc}
0 & \longrightarrow & (T + pP)/pP & \longrightarrow & P/pP & \longrightarrow & B/pB & \longrightarrow & 0 \\[6pt]
& & \Big\downarrow & & \Big\downarrow\vcenter{\rlap{$\scriptstyle\bar{f}$}} & & \Big\downarrow\vcenter{\rlap{$\scriptstyle\id$}} & & \\[6pt]
0 & \longrightarrow & (S + pQ)/pQ & \longrightarrow & Q/pQ & \longrightarrow & B/pB & \longrightarrow & 0
\end{array}
\]
Applying the Snake Lemma gives the exact sequence
\[
0 \longrightarrow \ker \bar{f} \longrightarrow \frac{T + pP}{pP} \longrightarrow \frac{S + pQ}{pQ} \longrightarrow \coker \bar{f} \longrightarrow 0,
\]
whose terms have dimensions
\[
m + a, \qquad \delta_p(P, T) = d, \qquad \delta_p(Q, S), \qquad b.
\]
From the injection $\ker \bar{f} \to (T + pP)/pP$, we see that
\[
m + a \le d, \qquad a < \infty,
\]
and by the additivity of dimension,
\[
\delta_p(Q, S) = d - (m + a) + b.
\]

\begin{lemma}\label{lem:3.3}
If $A$ is a torsion-free abelian group with $\rank A = r < \infty$, then
\[
\dim_{\Fp} A/pA \le r.
\]
\end{lemma}

\begin{proof}
Suppose there were $r + 1$ linearly independent cosets in $A/pA$, with representatives $a_1, \dots, a_{r+1} \in A$. Let $C = \langle a_1, \dots, a_{r+1} \rangle \le A$. Since $A$ is torsion-free, $C$ is a finitely generated torsion-free abelian group and hence free abelian of rank $s \le r$. Thus
\[
C \cong \Z^s, \qquad C/pC \cong (\Z/p\Z)^s,
\]
so $\dim_{\Fp}(C/pC) = s \le r$. Yet the image of $C/pC \to A/pA$ contains $r + 1$ linearly independent elements and thus has dimension at least $r + 1$, a contradiction. Therefore $\dim_{\Fp} A/pA \le r$.
\end{proof}

\begin{lemma}\label{lem:3.4}
If $\tau$ is a torsion abelian group with $\dim_{\Fp} \tau[p] = a < \infty$, then
\[
\dim_{\Fp} \tau/p\tau \le a.
\]
\end{lemma}

\begin{proof}
Let $\tau_p$ be the $p$-primary component of $\tau$, namely
\[
\tau_p = \{x \in \tau : p^n x = 0 \text{ for some } n \ge 1\}.
\]
Since $\tau$ is a torsion abelian group, $\tau = \bigoplus_{q \text{ prime}} \tau_q$. On $\tau_q$ for $q \neq p$, multiplication by $p$ is an automorphism, so
\[
\tau/p\tau \cong \tau_p/p\tau_p, \qquad \tau[p] = \tau_p[p],
\]
and it suffices to consider $\tau_p$.

If $\dim_{\Fp}(\tau_p/p\tau_p) > a$, choose $a + 1$ linearly independent cosets $x_1 + p\tau_p, \dots, x_{a+1} + p\tau_p$. Let $C = \langle x_1, \dots, x_{a+1} \rangle \le \tau_p$; since each generator has $p$-power order, $C$ is a finite abelian $p$-group. By the structure theorem for finite abelian groups,
\[
C \cong \bigoplus_{i=1}^s \Z/p^{n_i}\Z \qquad (n_i \ge 1).
\]
Since $(\Z/p^{n_i}\Z)/p(\Z/p^{n_i}\Z) \cong \Z/p\Z \cong \Fp$, we have
\[
\dim_{\Fp}(C/pC) = s = \dim_{\Fp} C[p] \le a.
\]
However, the image of $C/pC \longrightarrow \tau_p/p\tau_p$ contains $a + 1$ linearly independent elements and thus has dimension at least $a + 1$, a contradiction. Hence $\dim_{\Fp} \tau/p\tau \le a$.
\end{proof}

Now let $\tau = \Tor(D)$ be the torsion subgroup of $D = Q/M$, so that $D/\tau$ is torsion-free. Consider the short exact sequence
\[
0 \longrightarrow \tau \longrightarrow D \longrightarrow D/\tau \longrightarrow 0.
\]
Because $D/\tau$ is torsion-free, reducing modulo $p$ preserves exactness:
\[
0 \longrightarrow \tau/p\tau \longrightarrow D/pD \longrightarrow (D/\tau)/p(D/\tau) \longrightarrow 0.
\]
Hence
\[
b = \dim_{\Fp} D/pD = \dim_{\Fp} \tau/p\tau + \dim_{\Fp} (D/\tau)/p(D/\tau).
\]
Since $\rank(D/\tau) = \rank D = r$, Lemma \ref{lem:3.3} gives $\dim_{\Fp} (D/\tau)/p(D/\tau) \le r$; since $\tau[p] = D[p]$, Lemma \ref{lem:3.4} gives $\dim_{\Fp} \tau/p\tau \le a$. Thus $b \le a + r$, and consequently
\[
\delta_p(Q, S) = d - m - a + b \le d - m + r \le d,
\]
where $d, m, a, b, r$ are all finite integers.

Thus $\delta_p(Q, S) \le \delta_p(P, T)$. Repeating the same argument for $\Phi^{-1} : k(Q) \to k(P)$ yields the reverse inequality $\delta_p(P, T) \le \delta_p(Q, S)$, so in fact
\[
\delta_p(P, T) = \delta_p(Q, S).
\]
In particular, the chain of inequalities above must be equalities throughout, forcing $\rank D = r = m = \rank N$ (since $d - m + r = d$). This completes the proof of Proposition \ref{prop:3.1}.

\section{Finite Generation of \texorpdfstring{$Q/M$}{Q/M} and the Proof of Theorem 4.1}

Building on the equality $\rank N = \rank(Q/M) = m < \infty$ proved in Section 3, we now use finite field extension degrees to prove that the quotient group $Q/M$ is in fact a finitely generated abelian group. This allows us to split off a finite-rank direct summand from $Q$ and apply an exterior-power property for generating sets of finitely generated abelian groups to establish the weight-preserving isomorphism theorem when $\delta_p(P, T) < \infty$.

In this section, we prove the following theorem:

\begin{theorem}\label{thm:4.1}
If there exists a prime $p$ such that $\delta_p(P, T) < \infty$, then there exists a $B$-weight-preserving isomorphism
\[
\beta : P \xrightarrow{\sim} Q, \qquad \rho_Q \beta = \rho_P.
\]
\end{theorem}

If $m = \rank N = 0$, the conclusion already holds by Lemma \ref{lem:2.6}. Assume $m > 0$. By Proposition \ref{prop:3.1},
\[
\rank(Q/M) = \rank N = m.
\]
Choose $q_1, \dots, q_m \in Q$ whose images in $(Q/M) \otimes_{\Z} \Q$ form a $\Q$-basis, and set
\[
\Gamma = T_1 = M + \Z q_1 + \dots + \Z q_m, \qquad F_1 := F_R(Y^{q_1}, \dots, Y^{q_m}),
\]
where $F_R = \Phi(k(R))$ and $F = k(Q) = F_R(t_1, \dots, t_m)$. We first prove several auxiliary lemmas before completing the proof of Theorem \ref{thm:4.1}.

\begin{lemma}\label{lem:4.2}
The following statements hold:
\begin{enumerate}
\item[(i)] $Y^{q_1}, \dots, Y^{q_m}$ are algebraically independent over $F_R$;
\item[(ii)] $F/F_1$ is a finite field extension;
\item[(iii)] $\nu(F_1^\times) = \Gamma$.
\end{enumerate}
\end{lemma}

\begin{proof}
(i) Suppose to the contrary that there is a nonzero polynomial relation
\[
\sum_e a_e Y^{e_1 q_1 + \dots + e_m q_m} = 0,
\]
where $e = (e_1, \dots, e_m) \in \N^m$ and $a_e \in F_R$, with only finitely many $a_e \neq 0$. For each nonzero term,
\[
\nu\bigl(a_e Y^{e_1 q_1 + \dots + e_m q_m}\bigr) = \nu(a_e) + e_1 q_1 + \dots + e_m q_m.
\]
Since $\nu(F_R^\times) = M$, we have $\nu(a_e) \in M$, so the valuation of this term lies in the coset
\[
M + e_1 q_1 + \dots + e_m q_m.
\]
For distinct $e$, these cosets are distinct: otherwise
\[
M + \sum_{i=1}^m e_i q_i = M + \sum_{i=1}^m e_i' q_i \implies \sum_{i=1}^m (e_i - e_i') q_i \in M,
\]
which in $(Q/M) \otimes_{\Z} \Q$ gives $\sum_{i=1}^m (e_i - e_i')((q_i + M) \otimes 1) = 0$, forcing $e_i = e_i'$ for all $i$ by the $\Q$-linear independence of $q_1 + M, \dots, q_m + M$. Because distinct $e$ yield valuations in distinct $M$-cosets, the term of minimum valuation in the relation is unique and cannot be canceled, a contradiction. Thus $Y^{q_1}, \dots, Y^{q_m}$ are algebraically independent over $F_R$.

(ii) Since $F_1 = F_R(Y^{q_1}, \dots, Y^{q_m}) \subseteq F$, $F$ is an extension of $F_1$. By (i), $\trdeg_{F_R} F_1 = m$. Since $\trdeg_{F_R} F = m$, we have
\[
\trdeg_{F_1} F = \trdeg_{F_R} F - \trdeg_{F_R} F_1 = 0,
\]
so $F/F_1$ is algebraic. Since $F = F_R(t_1, \dots, t_m) = F_1(t_1, \dots, t_m)$ is finitely generated over $F_1$ by algebraic elements, $F/F_1$ is a finite extension.

(iii) We first show $\nu(F_1^\times) \subseteq \Gamma$. Any nonzero element of $F_1 = F_R(Y^{q_1}, \dots, Y^{q_m})$ is a quotient of polynomials of the form
\[
u_0 = \sum_e a_e Y^{e_1 q_1 + \dots + e_m q_m}, \qquad a_e \in F_R.
\]
By the argument in (i), terms with distinct $e$ have valuations in distinct $M$-cosets, so there is a unique term $e_0$ of minimum valuation, giving
\[
\nu(u_0) = \nu(a_{e_0}) + e_{0,1} q_1 + \dots + e_{0,m} q_m \in M + \sum_{i=1}^m \Z q_i = \Gamma.
\]
Applying the same to the denominator and using $\nu(u/v) = \nu(u) - \nu(v)$ gives $\nu(F_1^\times) \subseteq \Gamma$. Conversely, $M = \nu(F_R^\times) \subseteq \nu(F_1^\times)$ and $q_i = \nu(Y^{q_i}) \in \nu(F_1^\times)$, so $\Gamma \subseteq \nu(F_1^\times)$. Therefore $\nu(F_1^\times) = \Gamma$.
\end{proof}

\begin{lemma}\label{lem:4.3}
\[
[Q : \Gamma] \le [F : F_1] < \infty.
\]
\end{lemma}

\begin{proof}
Take representatives $u_1, \dots, u_s \in Q$ of any finitely many distinct cosets in $Q/\Gamma$. We claim that $Y^{u_1}, \dots, Y^{u_s}$ are linearly independent over $F_1$. Otherwise, there exist $a_1, \dots, a_s \in F_1$, not all zero, such that
\[
\sum_{i=1}^s a_i Y^{u_i} = 0.
\]
For each nonzero term, since $\nu(a_i) \in \Gamma$, we have $\nu(a_i Y^{u_i}) \in \Gamma + u_i$. Because $u_1, \dots, u_s$ represent distinct cosets modulo $\Gamma$, these valuations are pairwise distinct, so the sum has a unique term of minimum valuation and cannot vanish, a contradiction. Thus $Y^{u_1}, \dots, Y^{u_s}$ are linearly independent over $F_1$, whence $s \le [F : F_1] < \infty$, and therefore $[Q : \Gamma] \le [F : F_1] < \infty$.
\end{proof}

\begin{lemma}\label{lem:4.4}
There exist a finite-rank direct summand $Q_0$ of $Q$ and a complementary subgroup $Q_1$ such that
\[
Q = Q_0 \oplus Q_1, \qquad M = M_0 \oplus Q_1,
\]
where $M_0 = M \cap Q_0$, and
\[
\rho_Q(Q_0) = \rho_Q(M_0) =: B_0.
\]
\end{lemma}

\begin{proof}
By Lemma \ref{lem:4.3}, $\Gamma/M$ is finitely generated and $[Q : \Gamma] < \infty$, so $Q/M$ is finitely generated. Choose a basis of $Q$ such that the images of finitely many basis elements $q_1, \dots, q_k$ generate $Q/M$. Letting $Q_0$ be the subgroup spanned by these basis elements and $Q_1$ the subgroup spanned by the remaining basis elements, we have $Q = Q_0 \oplus Q_1$ with $Q_1 \le M$, and hence
\[
M = M_0 \oplus Q_1, \qquad M_0 = M \cap Q_0.
\]
For each basis element $q_i$ of $Q_0$, since $\rho_Q(M) = \rho_Q(f(P)) = B$, we may choose $m_i \in M$ such that $\rho_Q(m_i) = \rho_Q(q_i)$. Decomposing each $m_i$ in $M = M_0 \oplus Q_1$ involves only finitely many basis elements of $Q_1$; adjoining those finitely many basis elements to $Q_0$ and removing them from $Q_1$ preserves the finite-rank decomposition $Q = Q_0 \oplus Q_1$ and $M = M_0 \oplus Q_1$ while ensuring $m_i \in M_0$ for each $i$. Thus $\rho_Q(Q_0) \subseteq \rho_Q(M_0)$, and the reverse inclusion is obvious, giving $\rho_Q(Q_0) = \rho_Q(M_0) = B_0$.
\end{proof}

Let $n = \rank Q_0$. Since $Q/M \cong Q_0/M_0$ and $\rank(Q/M) = m$, we have $\rank M_0 = n - m$. Thus $N \oplus M_0$ and $Q_0$ are both free abelian groups of rank $n$ that map surjectively onto $B_0$ under their respective weight maps, whereas $B_0 = \rho_Q(M_0)$ is generated by at most $n - m < n$ elements.

\begin{lemma}\label{lem:4.5}
There exists a weight-preserving isomorphism
\[
\gamma : N \oplus M_0 \xrightarrow{\sim} Q_0.
\]
\end{lemma}

\begin{proof}
Choose $n$-element bases of $N \oplus M_0$ and $Q_0$; their weight images are both $n$-element generating sets of $B_0$. Since $B_0$ can be generated by fewer than $n$ elements, we have
\[
\bigwedge^n B_0 = 0 \qquad (\text{as one generator in any } n\text{-tuple is a linear combination of the remaining } n-1).
\]
By \cite[Lemma 2.4, Corollary 2.5]{ReichsteinYoussin2000}, if two surjections $\varphi, \psi : \Z^n \twoheadrightarrow B_0$ onto a finitely generated abelian group $B_0$ give two generating tuples $b_i = \varphi(e_i)$ and $c_i = \psi(e_i)$ satisfying $b_1 \wedge \dots \wedge b_n = \pm(c_1 \wedge \dots \wedge c_n)$ in $\bigwedge^n B_0$, then there exists an invertible matrix $U \in \GL_n(\Z)$ such that $\psi \circ U = \varphi$.

Here $b_1 \wedge \dots \wedge b_n = 0 = c_1 \wedge \dots \wedge c_n$ in $\bigwedge^n B_0 = 0$. Writing $\rho_{N \oplus M_0}(n, m_0) = \rho_Q(m_0)$, there therefore exists $U \in \GL_n(\Z)$ such that
\[
\rho_Q|_{Q_0} \circ U = \rho_{N \oplus M_0}.
\]
Taking $\gamma$ to be the isomorphism induced by $U$ completes the proof.
\end{proof}

\begin{proof}[Proof of Theorem \ref{thm:4.1} (continued)]
Since $P = N \oplus R$ and $f|_R : R \to M$ is an isomorphism, we have a weight-preserving isomorphism
\[
\begin{aligned}
\id_N \oplus (f|_R) =: \varphi : P &\xrightarrow{\sim} N \oplus M, \\
n + r &\longmapsto (n, f(r))
\end{aligned}
\]
(since $f$ is a weight-preserving homomorphism, its restriction $f|_R$ is a weight-preserving isomorphism onto $M$, and the weight map on $N$ is $0$). By $M = M_0 \oplus Q_1$, we have
\[
N \oplus M = (N \oplus M_0) \oplus Q_1.
\]
Applying Lemma \ref{lem:4.5} on $N \oplus M_0$ and $\id_{Q_1}$ on $Q_1$ gives
\[
P \cong_B (N \oplus M_0) \oplus Q_1 \cong_B Q_0 \oplus Q_1 = Q,
\]
and hence $P \cong_B Q$.
\end{proof}

\section{The Countable-Rank Case}

In this section, we assume that $P$ and $Q$ have at most countable rank. If $\delta_p(P, T)$ or $\delta_p(Q, S)$ is finite for some prime $p$, Theorem \ref{thm:4.1} already gives the conclusion; hence it suffices to consider the case where $\delta_p(P, T) = \delta_p(Q, S) = \aleph_0$ for every prime $p$. Using the Smith normal form on cofinite coordinate subgroups, we recursively construct primitive zero-weight vectors with pairwise disjoint supports to split off direct summands isomorphic to $\Z^{(\N)}$ from both $P$ and $Q$, and then construct a weight-preserving isomorphism between $P$ and $Q$ via the fiber product $W = P \times_B Q$.

\begin{lemma}\label{lem:5.1}
Write $P = \bigoplus_{i \in \N} \Z e_i$. Let $F \subseteq \N$ be a finite subset, and set
\[
P_F = \bigoplus_{i \in F} \Z e_i, \qquad P_{F^c} = \bigoplus_{i \notin F} \Z e_i, \qquad T_F = T \cap P_{F^c}.
\]
If the image of $T$ in $P/pP$ is infinite-dimensional, then the image of $T_F$ in $P/pP$ is also infinite-dimensional.
\end{lemma}

\begin{proof}
Let $\pi_F : P \to P_F$ be the coordinate projection. The image $\pi_F(T) \le P_F$ is a subgroup of a finitely generated free abelian group and is therefore finitely generated free abelian. Hence the short exact sequence
\[
0 \longrightarrow T_F \longrightarrow T \xrightarrow{\pi_F} \pi_F(T) \longrightarrow 0
\]
splits, giving $T = T_F \oplus C$ with $C \cong \pi_F(T)$ finitely generated. Consequently, the image of $C$ in $P/pP$ is finite-dimensional; if the image of $T_F$ in $P/pP$ were also finite-dimensional, the image of $T$ in $P/pP$ would be finite-dimensional, contradicting our hypothesis.
\end{proof}

\begin{lemma}\label{lem:5.2}
Under the hypotheses of Lemma \ref{lem:5.1} (holding for all primes $p$), for any finite subset $F \subseteq \N$ there exists $v \in T_F$ whose coordinates have greatest common divisor $1$; hence $v$ is a primitive vector in $P$ and can be extended to a basis of $P$.
\end{lemma}

\begin{proof}
Let $J \le \Z$ be the ideal generated by all coordinates of all vectors in $T_F$. If $J = d\Z$ with $d > 1$, we may choose a prime divisor $p \mid d$; then every coordinate of every vector in $T_F$ is divisible by $p$, so $T_F \subseteq pP$, contradicting Lemma \ref{lem:5.1}. Hence $J = \Z$.

Since $1 \in J$ is a finite $\Z$-linear combination of coordinates, there exist finitely many vectors $t_1, \dots, t_s \in T_F$ whose coordinates collectively have greatest common divisor $1$. Each $t_i$ has only finitely many nonzero coordinates, so $\supp(t_i)$ is finite. Set
\[
A = \bigcup_{i=1}^s \supp(t_i) \subseteq \N \setminus F.
\]
In the finite-rank free abelian group $P_A = \bigoplus_{i \in A} \Z e_i$, consider the subgroup generated by $t_1, \dots, t_s$. The matrix formed by the coordinates of $t_1, \dots, t_s$ has greatest common divisor of all its entries equal to $1$, which is its first invariant factor and appears as a ``$1$'' in the upper-left corner of its Smith normal form. Therefore, the subgroup generated by $t_1, \dots, t_s$ contains a vector $v \in T_F$ that is primitive in $P_A$; since $v$ has zero coordinates outside $A$, $v$ is also primitive in $P$.
\end{proof}

\begin{proposition}\label{prop:5.3}
If $P$ has at most countable rank and satisfies $\delta_p(P, T) = \aleph_0$ for all primes $p$, then there exists a direct sum decomposition
\[
P = P_0 \oplus Z_0, \qquad Z_0 \cong \Z^{(\N)}, \qquad Z_0 \le T.
\]
\end{proposition}

\begin{proof}
By Lemma \ref{lem:5.2}, for any finite subset $F \subseteq \N$ there exists a primitive vector $v \in T_F$. We recursively construct a sequence of primitive vectors $v_1, v_2, \dots \in T$ with pairwise disjoint supports:
\begin{itemize}
\item First take $F_0 = \varnothing$; by Lemma \ref{lem:5.2}, there exists a primitive vector $v_1 \in T$, and $A_1 = \supp(v_1)$ is finite.
\item Suppose $v_1, \dots, v_{n-1} \in T$ have been constructed with finite, pairwise disjoint supports $A_i = \supp(v_i)$ ($1 \le i \le n-1$). Set $F_{n-1} = \bigcup_{i=1}^{n-1} A_i$, which is finite; by Lemma \ref{lem:5.2}, we may choose a primitive vector $v_n \in T_{F_{n-1}}$, whose support $A_n = \supp(v_n)$ is finite and satisfies $A_n \cap F_{n-1} = \varnothing$.
\end{itemize}
This produces primitive vectors $v_1, v_2, \dots \in T$ whose supports $A_i = \supp(v_i)$ are finite and pairwise disjoint.

For each $i \ge 1$, set $P_{A_i} = \bigoplus_{j \in A_i} \Z e_j$. Since $v_i$ is primitive in $P$ and supported in $A_i$, $v_i$ is primitive in $P_{A_i}$ and thus extends to a basis of $P_{A_i}$, yielding a complementary subgroup $C_i \le P_{A_i}$ such that $P_{A_i} = \Z v_i \oplus C_i$. Since the subsets $A_i$ are pairwise disjoint, we have
\[
\begin{aligned}
P &= \left(\bigoplus_{i \ge 1} P_{A_i}\right) \oplus \left(\bigoplus_{j \notin \bigcup_{i \ge 1} A_i} \Z e_j\right) = \left(\bigoplus_{i \ge 1} (\Z v_i \oplus C_i)\right) \oplus \left(\bigoplus_{j \notin \bigcup_{i \ge 1} A_i} \Z e_j\right) \\
&= \left(\bigoplus_{i \ge 1} \Z v_i\right) \oplus \left(\bigoplus_{i \ge 1} C_i \oplus \bigoplus_{j \notin \bigcup_{i \ge 1} A_i} \Z e_j\right).
\end{aligned}
\]
Setting $Z_0 = \bigoplus_{i \ge 1} \Z v_i$ and $P_0 = \left(\bigoplus_{i \ge 1} C_i\right) \oplus \left(\bigoplus_{j \notin \bigcup_{i \ge 1} A_i} \Z e_j\right)$ gives $P = P_0 \oplus Z_0$, with $Z_0 \le T$ and $Z_0 \cong \bigoplus_{i \ge 1} \Z \cong \Z^{(\N)}$.
\end{proof}

\begin{theorem}\label{lem:5.4}
If $P$ and $Q$ have at most countable rank and there exists a $B$-weighted isomorphism $\Phi : k(P) \xrightarrow{\sim} k(Q)$, then there exists a weight-preserving group isomorphism
\[
\beta : P \xrightarrow{\sim} Q, \qquad \rho_Q \beta = \rho_P.
\]
\end{theorem}

\begin{proof}
It suffices to consider the case where $\delta_p(P, T) = \delta_p(Q, S) = \aleph_0$ for all primes $p$. By Proposition \ref{prop:5.3}, there exist decompositions $P = P_0 \oplus Z_P$ and $Q = Q_0 \oplus Z_Q$ with $Z_P \cong Z_Q \cong \Z^{(\N)}$, $Z_P \le T = \ker \rho_P$, and $Z_Q \le S = \ker \rho_Q$.

Define the fiber product
\[
W = \{(p, q) \in P \oplus Q : \rho_P(p) = \rho_Q(q)\}.
\]
The projection $\pi_P : W \longrightarrow P$, $(p, q) \mapsto p$, is surjective: for any $p \in P$, since $\rho_Q$ is surjective onto $B$, there exists $q \in Q$ with $\rho_Q(q) = \rho_P(p)$, so $(p, q) \in W$. The kernel of $\pi_P : W \to P$ is clearly $0 \oplus S \cong S$. Since $P$ is free abelian, the short exact sequence
\[
0 \longrightarrow S \longrightarrow W \xrightarrow{\pi_P} P \longrightarrow 0
\]
splits. Equip $W$ with the weight map $\rho_W(p, q) = \rho_P(p) = \rho_Q(q)$. Any section $s : P \to W$ satisfying $\pi_P \circ s = \id_P$ automatically preserves weights because $\rho_W(s(p)) = \rho_P(\pi_P(s(p))) = \rho_P(p)$, while $S \hookrightarrow W$ has weight identically $0$. Thus
\[
\begin{aligned}
P \oplus S &\xrightarrow{\sim} W, \\
(p, t) &\longmapsto s(p) + (0, t)
\end{aligned}
\]
is a weight-preserving isomorphism, and by symmetry $W \cong_B P \oplus S \cong_B Q \oplus T$.

Since $\delta_p(P, T) = \delta_p(Q, S) = \aleph_0$ and $P, Q$ have at most countable rank, $T$ and $S$ are free abelian groups of countably infinite rank with weight identically $0$. Because $Z_P, S \cong \Z^{(\N)}$ both have weight $0$, we have $Z_P \oplus S \cong_B \Z^{(\N)} \cong_B Z_P$, and hence
\[
P \oplus S \cong_B P_0 \oplus Z_P \oplus S \cong_B P_0 \oplus Z_P = P.
\]
Similarly $Q \oplus T \cong_B Q$, and therefore
\[
P \cong_B W \cong_B Q.
\]
\end{proof}

\section{The Arbitrary-Rank Case}

In this section, we extend the countable-rank weight-preserving isomorphism theorem of Section 5 to free abelian groups $P$ and $Q$ of arbitrary rank. Because every rational function involves only finitely many basis variables, a countable back-and-forth closure allows us to construct a continuous increasing chain of index sets $(I_\alpha, J_\alpha)_{\alpha \le \Lambda}$ whose new free direct summands at each successor stage have at most countable rank, and we extend the weight-preserving isomorphism stage by stage via the countable-rank theorem and block upper-triangular corrections.

In this section, we prove the main commutative theorem:

\begin{theorem}\label{thm:6.1}
Let $P$ and $Q$ be free abelian groups of arbitrary rank, and let $\rho_P : P \twoheadrightarrow B$ and $\rho_Q : Q \twoheadrightarrow B$ be surjective homomorphisms. If there exists a $B$-weighted $k$-field isomorphism
\[
\Phi : k(P) \xrightarrow{\sim} k(Q),
\]
then there exists a weight-preserving group isomorphism
\[
\beta : P \xrightarrow{\sim} Q, \qquad \rho_Q \beta = \rho_P.
\]
\end{theorem}

Choose bases $(e_i)_{i \in I}$ of $P$ and $(f_j)_{j \in J}$ of $Q$, and set $x_i = X^{e_i}$ and $y_j = \Phi^{-1}(Y^{f_j}) \in k(P)$. Then
\[
k(P) = k(x_i : i \in I) = k(y_j : j \in J),
\]
and because $\Phi$ preserves weights, both generating sets on $k(P)$ are acted upon by the same character automorphisms $\sigma_\chi^P$.

\begin{lemma}\label{lem:6.2}
Suppose $I_0 \subseteq I$ and $J_0 \subseteq J$ satisfy
\[
k(x_i : i \in I_0) = k(y_j : j \in J_0).
\]
Given a coordinate $x_{i_*}$ or $y_{j_*}$ not lying in $k(x_i : i \in I_0) = k(y_j : j \in J_0)$, there exist subsets $I_1 \supseteq I_0$ and $J_1 \supseteq J_0$ such that:
\begin{enumerate}
\item[(i)] $i_* \in I_1$ (or $j_* \in J_1$);
\item[(ii)] $I_1 \setminus I_0$ and $J_1 \setminus J_0$ are at most countable;
\item[(iii)] $k(x_i : i \in I_1) = k(y_j : j \in J_1)$.
\end{enumerate}
\end{lemma}

\begin{proof}
We give the proof for $x_{i_*}$ (the case of $y_{j_*}$ is identical). For each $i \in I$, choose a finite subset $F(i) \subseteq J$ such that $x_i \in k(y_j : j \in F(i))$; for example, if $x_i = \frac{y_2 + y_5}{y_8}$, we may take $F(i) = \{2, 5, 8\}$. Similarly, for each $j \in J$, choose a finite subset $G(j) \subseteq I$ such that $y_j \in k(x_i : i \in G(j))$. The fact that any rational function involves only finitely many variables guarantees the existence of $F(i)$ and $G(j)$.

Define recursively $S_0 = \{i_*\}$, $T_0 = \varnothing$, and for $n \ge 0$:
\[
\begin{aligned}
T_{n+1} &= T_n \cup \bigcup_{i \in S_n} (F(i) \setminus J_0), \\
S_{n+1} &= S_n \cup \bigcup_{j \in T_{n+1}} (G(j) \setminus I_0).
\end{aligned}
\]
In other words, $T_{n+1}$ adjoins to $T_n$ all new $y_j$-indices outside $J_0$ needed to express the variables $x_i$ appearing in $S_n$, while $S_{n+1}$ adjoins to $S_n$ all new $x_i$-indices outside $I_0$ needed to express the variables $y_j$ appearing in $T_{n+1}$.

Since each $F(i)$ and $G(j)$ is finite, each $S_n$ and $T_n$ is finite, so the unions
\[
S = \bigcup_{n=0}^\infty S_n, \qquad T' = \bigcup_{n=0}^\infty T_n
\]
are at most countable. Setting $I_1 = I_0 \cup S$ and $J_1 = J_0 \cup T'$, we see that $I_1 \setminus I_0$ and $J_1 \setminus J_0$ are at most countable and $i_* \in I_1$. By construction, $F(i) \subseteq J_1$ for every $i \in S$, so $x_i \in k(y_j : j \in J_1)$ for all $i \in S$ as well as for all $i \in I_0$; hence
\[
k(x_i : i \in I_1) \subseteq k(y_j : j \in J_1).
\]
By the same token, $k(y_j : j \in J_1) \subseteq k(x_i : i \in I_1)$, so the two fields are equal.
\end{proof}

\begin{lemma}\label{lem:6.3}
There exist an ordinal $\Lambda$ and a continuous increasing family of index pairs $(I_\alpha, J_\alpha)_{\alpha \le \Lambda}$ such that:
\begin{enumerate}
\item[(i)] $I_0 = J_0 = \varnothing$, $I_\Lambda = I$, and $J_\Lambda = J$;
\item[(ii)] for every $\alpha \le \Lambda$,
\[
K_\alpha := k(x_i : i \in I_\alpha) = k(y_j : j \in J_\alpha);
\]
\item[(iii)] at each successor stage, $I_{\alpha+1} \setminus I_\alpha$ and $J_{\alpha+1} \setminus J_\alpha$ are at most countable;
\item[(iv)] at limit ordinals $\lambda$, $I_\lambda = \bigcup_{\alpha < \lambda} I_\alpha$ and $J_\lambda = \bigcup_{\alpha < \lambda} J_\alpha$.
\end{enumerate}
\end{lemma}

\begin{proof}
Well-order $I \sqcup J$. Starting from $I_0 = J_0 = \varnothing$, if at stage $\alpha$ there remain coordinates in $I \sqcup J$ not yet included in $I_\alpha \sqcup J_\alpha$, pick the least such coordinate and apply Lemma \ref{lem:6.2} to obtain $(I_{\alpha+1}, J_{\alpha+1})$. At a limit ordinal $\lambda$, define
\[
I_\lambda = \bigcup_{\alpha < \lambda} I_\alpha, \qquad J_\lambda = \bigcup_{\alpha < \lambda} J_\alpha.
\]
Because the chains are increasing and every rational function involves only finitely many variables, we have
\[
k(x_i : i \in I_\lambda) = \bigcup_{\alpha < \lambda} k(x_i : i \in I_\alpha) = \bigcup_{\alpha < \lambda} k(y_j : j \in J_\alpha) = k(y_j : j \in J_\lambda).
\]
Since at least the least omitted coordinate is added at each successor step, the process exhausts $I \sqcup J$ at some ordinal $\Lambda \le |I \sqcup J|$.
\end{proof}

Let $P_\alpha \le P$ and $Q_\alpha \le Q$ be the subgroups generated by the basis elements indexed by $I_\alpha$ and $J_\alpha$, respectively, and define
\[
B_\alpha^P = \rho_P(P_\alpha), \qquad B_\alpha^Q = \rho_Q(Q_\alpha).
\]

\begin{lemma}\label{lem:6.4}
For every $\alpha \le \Lambda$, we have $B_\alpha^P = B_\alpha^Q =: B_\alpha$.
\end{lemma}

\begin{proof}
Suppose $b \in B_\alpha^P \setminus B_\alpha^Q$. By Lemma \ref{lem:2.1}, there exists a character $\chi : B \to k_*^\times$ such that $\chi|_{B_\alpha^Q} = 1$ and $\chi(b) \neq 1$. Then $\sigma_\chi^Q$ fixes $Y^{f_j}$ for all $j \in J_\alpha$. For $y_j = \Phi^{-1}(Y^{f_j})$ ($j \in J_\alpha$), we have
\[
\sigma_\chi^P(y_j) = \sigma_\chi^P\bigl(\Phi^{-1}(Y^{f_j})\bigr) = \Phi^{-1}\bigl(\sigma_\chi^Q(Y^{f_j})\bigr) = \Phi^{-1}(Y^{f_j}) = y_j.
\]
Hence $\sigma_\chi^P$ fixes $y_j$ for all $j \in J_\alpha$, and therefore fixes the subfield $K_\alpha = k(y_j : j \in J_\alpha)$. However, choosing $p \in P_\alpha$ with $\rho_P(p) = b$, we have $X^p \in k(x_i : i \in I_\alpha) = K_\alpha$, so $\sigma_\chi^P(X^p) = X^p$, whereas $\sigma_\chi^P(X^p) = \chi(\rho_P(p))X^p = \chi(b)X^p \neq X^p$, a contradiction! Thus $B_\alpha^P \subseteq B_\alpha^Q$, and by symmetry $B_\alpha^Q \subseteq B_\alpha^P$, giving $B_\alpha^P = B_\alpha^Q$.
\end{proof}

Suppose a weight-preserving isomorphism $\beta_\alpha : P_\alpha \xrightarrow{\sim} Q_\alpha$ has been constructed. At the successor stage $\alpha + 1$, let $P_\alpha^{\mathrm{new}}$ and $Q_\alpha^{\mathrm{new}}$ be the subgroups spanned by the basis elements in $I_{\alpha+1} \setminus I_\alpha$ and $J_{\alpha+1} \setminus J_\alpha$, respectively; then
\[
P_{\alpha+1} = P_\alpha \oplus P_\alpha^{\mathrm{new}}, \qquad Q_{\alpha+1} = Q_\alpha \oplus Q_\alpha^{\mathrm{new}},
\]
where $P_\alpha^{\mathrm{new}}$ and $Q_\alpha^{\mathrm{new}}$ have at most countable rank.

\begin{lemma}\label{lem:6.5}
Every character
\[
\bar{\chi} : B_{\alpha+1}/B_\alpha \longrightarrow k_*^\times
\]
extends to a character $\chi : B \to k_*^\times$ satisfying $\chi|_{B_\alpha} = 1$. Consequently, the action of $\chi$ fixes $K_\alpha$ pointwise and induces on $P_\alpha^{\mathrm{new}}$ and $Q_\alpha^{\mathrm{new}}$ the weighted action associated with their quotient weight maps onto $B_{\alpha+1}/B_\alpha$.
\end{lemma}

\begin{proof}
Since $k_*^\times$ is divisible and hence injective in the category of abelian groups, $\bar{\chi}$ extends to $B/B_\alpha$, and composing with $B \twoheadrightarrow B/B_\alpha$ gives $\chi$. Because $\chi|_{B_\alpha} = 1$, $\sigma_\chi^P$ fixes all elements of $P_\alpha$ and $Q_\alpha$ and hence fixes $K_\alpha$.
\end{proof}

Since $B_{\alpha+1} = B_\alpha + \rho_P(P_\alpha^{\mathrm{new}}) = B_\alpha + \rho_Q(Q_\alpha^{\mathrm{new}})$, the composite quotient weight maps $\bar{\rho}_P : P_\alpha^{\mathrm{new}} \twoheadrightarrow B_{\alpha+1}/B_\alpha$ and $\bar{\rho}_Q : Q_\alpha^{\mathrm{new}} \twoheadrightarrow B_{\alpha+1}/B_\alpha$ (obtained by composing $\rho_P, \rho_Q$ with $\pi_\alpha : B_{\alpha+1} \twoheadrightarrow B_{\alpha+1}/B_\alpha$) are both surjective. By $K_\alpha(P_\alpha^{\mathrm{new}}) = K_\alpha(Q_\alpha^{\mathrm{new}}) = K_{\alpha+1}$ and Lemma \ref{lem:6.5}, the identity map on $K_{\alpha+1}$ (via $\Phi$) is a $(B_{\alpha+1}/B_\alpha)$-weighted $K_\alpha$-field isomorphism. By Theorem \ref{lem:5.4}, there exists a quotient-weight-preserving isomorphism
\[
\bar{\beta}_\alpha : P_\alpha^{\mathrm{new}} \xrightarrow{\sim} Q_\alpha^{\mathrm{new}}, \qquad \bar{\rho}_Q \circ \bar{\beta}_\alpha = \bar{\rho}_P.
\]
Consequently, for every $r \in P_\alpha^{\mathrm{new}}$, the difference $d_\alpha(r) := \rho_P(r) - \rho_Q(\bar{\beta}_\alpha(r))$ lies in $B_\alpha$, defining a homomorphism $d_\alpha : P_\alpha^{\mathrm{new}} \to B_\alpha$. Since $P_\alpha^{\mathrm{new}}$ is free abelian and $\rho_Q|_{Q_\alpha} : Q_\alpha \twoheadrightarrow B_\alpha$ is surjective, we can lift $d_\alpha$ basis element by basis element to a homomorphism
\[
h_\alpha : P_\alpha^{\mathrm{new}} \longrightarrow Q_\alpha, \qquad \rho_Q \circ h_\alpha = d_\alpha.
\]
Define $\beta_{\alpha+1} : P_{\alpha+1} \to Q_{\alpha+1}$ by
\[
\beta_{\alpha+1}(p + r) = \beta_\alpha(p) + h_\alpha(r) + \bar{\beta}_\alpha(r), \qquad p \in P_\alpha, \ r \in P_\alpha^{\mathrm{new}}.
\]

\begin{lemma}\label{lem:6.6}
$\beta_{\alpha+1}$ is a weight-preserving isomorphism extending $\beta_\alpha$.
\end{lemma}

\begin{proof}
With respect to the direct sum decompositions $P_\alpha \oplus P_\alpha^{\mathrm{new}}$ and $Q_\alpha \oplus Q_\alpha^{\mathrm{new}}$, $\beta_{\alpha+1}$ has the upper-triangular block matrix
\[
\begin{pmatrix}
\beta_\alpha & h_\alpha \\
0 & \bar{\beta}_\alpha
\end{pmatrix}.
\]
Since both diagonal blocks $\beta_\alpha$ and $\bar{\beta}_\alpha$ are isomorphisms, $\beta_{\alpha+1}$ is an isomorphism with explicit inverse
\[
q + s \longmapsto \beta_\alpha^{-1}\bigl(q - h_\alpha(\bar{\beta}_\alpha^{-1}(s))\bigr) + \bar{\beta}_\alpha^{-1}(s), \qquad q \in Q_\alpha, \ s \in Q_\alpha^{\mathrm{new}}.
\]
Clearly $\beta_{\alpha+1}|_{P_\alpha} = \beta_\alpha$. Moreover,
\[
\begin{aligned}
\rho_Q(\beta_{\alpha+1}(p + r)) &= \rho_Q(\beta_\alpha(p)) + \rho_Q(h_\alpha(r)) + \rho_Q(\bar{\beta}_\alpha(r)) \\
&= \rho_P(p) + d_\alpha(r) + \rho_Q(\bar{\beta}_\alpha(r)) \\
&= \rho_P(p) + \rho_P(r) = \rho_P(p + r),
\end{aligned}
\]
so $\beta_{\alpha+1}$ preserves weights.
\end{proof}

\begin{proof}[Proof of Theorem \ref{thm:6.1}]
Starting from the zero map $\beta_0 : 0 \to 0$, extend $\beta_\alpha$ to $\beta_{\alpha+1}$ at each successor stage via Lemma \ref{lem:6.6}, and at each limit ordinal $\lambda$ set
\[
\beta_\lambda = \bigcup_{\alpha < \lambda} \beta_\alpha.
\]
Because each stage extends all earlier stages, $\beta_\lambda$ is a well-defined homomorphism; it is injective because any nonzero element belongs to $P_\alpha$ for some $\alpha < \lambda$, surjectivity holds because any $q \in Q_\lambda$ involves only finitely many coordinates and thus belongs to $Q_\alpha$ for some $\alpha < \lambda$, and weight preservation is maintained at every stage. At stage $\Lambda$, we obtain a weight-preserving isomorphism $\beta = \beta_\Lambda : P \xrightarrow{\sim} Q$ satisfying $\rho_Q \beta = \rho_P$.
\end{proof}

\section{The Center of the Division Ring of Fractions of a Torsion-Free Nilpotent Group}

Starting in this section, we turn to the division ring reconstruction problem for torsion-free nilpotent groups. Using a bi-invariant total ordering on a torsion-free nilpotent group $G$ and the embedding into the Malcev--Neumann formal power series division ring, we prove that the center of the division ring of fractions $D_{\Q}(G)$ coincides with the fraction field $\Frac \Q[Z(G)]$ of the rational group algebra of the center $Z(G)$, and establish the isolation of $Z(G)$ in $G$.

Let $G$ be a torsion-free nilpotent group, written multiplicatively. By \cite{Malcev1951,Neumann1949}, $G$ admits a bi-invariant total ordering $<$, so the product of the lowest terms of two nonzero elements in $\Q[G]$ is the lowest term of their product, and thus $\Q[G]$ has no zero divisors. Since $G$ is nilpotent and hence amenable, by \cite{Tamari1954,BartholdiKielak2019}, $\Q[G]$ is an Ore domain and possesses a division ring of fractions $D_{\Q}(G)$. Embedding $\Q[G]$ into the Malcev--Neumann division ring $\Q((G, <))$, where every nonzero element is invertible, the universal property of Ore fractions yields a unique embedding
\[
D_{\Q}(G) \hookrightarrow \Q((G, <)).
\]
Hereafter, the support $\supp(z)$ of an element $z \in D_{\Q}(G)$ refers to its Malcev--Neumann support under this embedding.

\begin{lemma}\label{lem:7.1}
If $0 \neq z \in Z(D_{\Q}(G))$, then
\[
\supp(z) \subseteq Z(G).
\]
\end{lemma}

\begin{proof}
For any $h \in G$, since $z$ is central, we have $hzh^{-1} = z$. By uniqueness of the Malcev--Neumann series expansion $z = \sum_{g \in G} a_g g$ ($a_g \in \Q$), we have
\[
hzh^{-1} = \sum_{g \in G} a_g (hgh^{-1}),
\]
and since conjugation $g \mapsto hgh^{-1}$ is a bijection on $G$,
\[
\supp(hzh^{-1}) = \{hgh^{-1} : a_g \neq 0\} = h\,\supp(z)h^{-1} = \supp(z).
\]
If there were $g \in \supp(z)$ with $g \notin Z(G)$, there would exist $h \in G$ such that $hgh^{-1} \neq g$. Since $<$ is a bi-invariant total ordering on $G$, the inner automorphism $c_h : x \mapsto hxh^{-1}$ is order-preserving. If $c_h(g) > g$, repeatedly applying $c_h^{-1}$ produces an infinite strictly decreasing chain
\[
g > c_h^{-1}(g) > c_h^{-2}(g) > \cdots
\]
all lying in $\supp(z)$, which contradicts the well-ordering of the support of a Malcev--Neumann series. If $c_h(g) < g$, applying $c_h$ repeatedly gives the same contradiction.
\end{proof}

\begin{proposition}\label{prop:7.2}
For any torsion-free nilpotent group $G$,
\[
Z(D_{\Q}(G)) = \Frac \Q[Z(G)].
\]
\end{proposition}

\begin{proof}
The right-hand side is clearly contained in the center. Conversely, take $0 \neq z \in Z(D_{\Q}(G))$, and write $z = ab^{-1}$ with $0 \neq a, b \in \Q[G]$, so $zb = a$ in $\Q((G, <))$. Set $C = G/Z(G)$ and decompose the finite sums $a$ and $b$ along $C$-cosets:
\[
a = \sum_{c \in C} a_c, \qquad b = \sum_{c \in C} b_c,
\]
where only finitely many components are nonzero. Fixing a coset representative $g_c \in G$ for each coset that appears, write
\[
a_c = \alpha_c g_c, \qquad b_c = \beta_c g_c, \qquad \alpha_c, \beta_c \in \Q[Z(G)].
\]
By Lemma \ref{lem:7.1}, $\supp(z) \subseteq Z(G)$, so left multiplication by $z$ preserves each $G/Z(G)$-coset. Thus $zb = a$ implies $zb_c = a_c$ for each coset $c \in C$. Choosing any $c$ with $b_c \neq 0$, we have $a_c = zb_c \neq 0$, and hence
\[
z = a_c b_c^{-1} = \alpha_c g_c (\beta_c g_c)^{-1} = \alpha_c \beta_c^{-1} \in \Frac \Q[Z(G)],
\]
since $\alpha_c$ and $\beta_c$ are central.
\end{proof}

\begin{lemma}\label{lem:7.3}
If a group $G$ is bi-orderable and $x^n = y^n$ for some $n \ge 1$, then $x = y$.
\end{lemma}

\begin{proof}
If $x < y$, then multiplying on the left and right preserves strict inequalities, yielding
\[
x^n = x^{n-1}x < x^{n-1}y < x^{n-2}y^2 < \dots < y^n,
\]
a contradiction. The case $y < x$ is analogous.
\end{proof}

\begin{lemma}\label{lem:7.4}
If $G$ is torsion-free nilpotent, then $Z(G)$ is isolated in $G$; that is, $g^n \in Z(G)$ ($n \ge 1$) implies $g \in Z(G)$. In particular, $G/Z(G)$ is torsion-free nilpotent.
\end{lemma}

\begin{proof}
By \cite{Malcev1951}, $G$ is bi-orderable and thus has unique roots by Lemma \ref{lem:7.3}. If $g^n \in Z(G)$, then for every $h \in G$,
\[
(hgh^{-1})^n = hg^n h^{-1} = g^n,
\]
so uniqueness of roots implies $hgh^{-1} = g$, meaning $g \in Z(G)$. Thus $G/Z(G)$ is torsion-free, and it is nilpotent as a quotient of a nilpotent group.
\end{proof}

\section{Central Localization and the Leading-Term Group}

In this section, we use the strong grading of the central localization $A_G = (\Q[Z(G)] \setminus \{0\})^{-1}\Q[G]$ over the quotient group $G/Z(G)$ to define a lowest-coset leading-term homomorphism $\ell : D^\times \to F^\times G$ that preserves coefficients in the central field $F$. Using $\ell$, we prove $H/Z(H) \cong G/Z(G)$ and replace $\varphi(H)$ by an isomorphic subgroup $H_1 \le F^\times G$ satisfying $F^\times H_1 = F^\times G$ and $D_{\Q}(H_1) = D$.

Suppose there is a division ring isomorphism
\[
\varphi : D_{\Q}(H) \xrightarrow{\sim} D_{\Q}(G).
\]
By Proposition \ref{prop:7.2}, $\varphi$ maps the center of $D_{\Q}(H)$ isomorphically onto the center of $D := D_{\Q}(G)$, so we set
\[
F = Z(D) = \Frac \Q[Z(G)] = \varphi\bigl(\Frac \Q[Z(H)]\bigr),
\]
and view $\varphi(H)$ as a subgroup of $D^\times$.

\begin{lemma}\label{lem:8.1}
Let $R$ be a right Ore domain with classical right division ring of fractions $Q_{\mathrm{cl}}(R)$, and let $\Gamma$ be any group. If a map
\[
\lambda_0 : R \setminus \{0\} \longrightarrow \Gamma
\]
satisfies $\lambda_0(ab) = \lambda_0(a)\lambda_0(b)$, then there exists a unique group homomorphism
\[
\lambda : Q_{\mathrm{cl}}(R)^\times \longrightarrow \Gamma
\]
satisfying $\lambda(as^{-1}) = \lambda_0(a)\lambda_0(s)^{-1}$.
\end{lemma}

\begin{proof}
If $as^{-1} = bt^{-1}$, there exist nonzero $u, v \in R$ such that $au = bv$ and $su = tv$, so
\[
\lambda_0(a)\lambda_0(s)^{-1} = \lambda_0(au)\lambda_0(su)^{-1} = \lambda_0(bv)\lambda_0(tv)^{-1} = \lambda_0(b)\lambda_0(t)^{-1},
\]
showing that $\lambda$ is well-defined. For multiplicativity, given $x = as^{-1}$ and $y = bt^{-1}$, choose nonzero $b', s' \in R$ with $sb' = bs'$, so $s^{-1}b = b'(s')^{-1}$ and $xy = ab'(ts')^{-1}$. From $sb' = bs'$ we have $\lambda_0(b')\lambda_0(s')^{-1} = \lambda_0(s)^{-1}\lambda_0(b)$, whence
\[
\lambda(xy) = \lambda_0(a)\lambda_0(s)^{-1}\lambda_0(b)\lambda_0(t)^{-1} = \lambda(x)\lambda(y).
\]
Uniqueness is immediate since every nonzero fraction has the form $as^{-1}$.
\end{proof}

By Lemma \ref{lem:7.4}, $C := G/Z(G)$ is torsion-free nilpotent and thus admits a bi-invariant total ordering $<$. Localizing $\Q[G]$ at the nonzero elements of the center gives
\[
A_G := (\Q[Z(G)] \setminus \{0\})^{-1}\Q[G].
\]
For each $c \in C$, choose a coset representative $g_c \in G$ with $g_{1_C} = 1_G$, and write $Z = Z(G)$. Since $G = \bigsqcup_{c \in C} Zg_c$, as a $\Q[Z]$-module we have
\[
\Q[G] = \bigoplus_{c \in C} \Q[Z] g_c.
\]
Localizing with respect to the multiplicatively closed set $\Q[Z] \setminus \{0\}$ yields
\[
A_G = (\Q[Z] \setminus \{0\})^{-1}\left(\bigoplus_{c \in C} \Q[Z] g_c\right) = \bigoplus_{c \in C} F g_c = \bigoplus_{c \in C} A_c,
\]
where $A_c := F g_c$. Thus every $a \in A_G$ has a unique finite expansion
\[
a = \sum_{c \in E} f_c g_c, \qquad E \subseteq C \text{ finite}, \ f_c \in F.
\]
For any $c, d \in C$, $g_c g_d$ and $g_{cd}$ lie in the same $Z$-coset, so $g_c g_d = z_{c,d} g_{cd}$ for some $z_{c,d} \in Z \subseteq F^\times$. Because $Z$ commutes with $F$ and $Fz_{c,d} = F$, we have
\[
A_c A_d = (F g_c)(F g_d) = F(g_c g_d) = F z_{c,d} g_{cd} = F g_{cd} = A_{cd},
\]
with $A_{1_C} = F$. Hence $A_G$ is a strongly $C$-graded ring.

For $0 \neq a = \sum_{c \in E} f_c g_c \in A_G$ ($f_c \in F^\times$) with $c_0 = \min E$, define
\[
\ell_0(a) := f_{c_0} g_{c_0} \in F^\times G.
\]

\begin{lemma}\label{lem:8.2}
$\ell_0 : A_G \setminus \{0\} \to F^\times G$ is a multiplicative map.
\end{lemma}

\begin{proof}
If $a$ and $b$ have minimum cosets $c_0$ and $d_0$, respectively, the bi-orderability of $C$ ensures that the minimum graded component of $ab$ arises uniquely as the product of their lowest terms, so $\ell_0(ab) = \ell_0(a)\ell_0(b)$.
\end{proof}

If $R$ is a right Ore domain and $S \subseteq Z(R) \setminus \{0\}$ is multiplicatively closed, then $S^{-1}R$ remains a right Ore domain (for $\frac{a}{s}, \frac{b}{t} \in S^{-1}R$, choosing $ar = bu$ in $R$ gives $\frac{a}{s}\cdot sr = \frac{b}{t}\cdot tu$) and $Q_{\mathrm{cl}}(S^{-1}R) = Q_{\mathrm{cl}}(R)$. Thus $A_G$ is a right Ore domain with $Q_{\mathrm{cl}}(A_G) = D$. By Lemma \ref{lem:8.1}, $\ell_0$ extends to a group homomorphism
\[
\ell : D^\times \longrightarrow F^\times G, \qquad \ell|_{F^\times G} = \id.
\]
Let $\nu : D^\times \xrightarrow{\ell} F^\times G \twoheadrightarrow G/Z(G) = C$ be the composite homomorphism mapping $f_{c_0} g_{c_0} \in F^\times G$ to the $Z$-coset of $g_{c_0}$.

\begin{lemma}\label{lem:8.3}
If $N \triangleleft K$ with $N \neq 1$ and $K$ is a nilpotent group, then $N \cap Z(K) \neq 1$.
\end{lemma}

\begin{proof}
Let $1 = Z_0(K) < Z_1(K) < \dots < Z_c(K) = K$ be the upper central series of $K$, and let $i \ge 1$ be minimal such that $N \cap Z_i(K) \neq 1$. If $i = 1$, we are done. If $i > 1$, take $1 \neq x \in N \cap Z_i(K)$; if $x \notin Z(K)$, there exists $g \in K$ with $1 \neq [x, g] \in N \cap Z_{i-1}(K)$, contradicting the minimality of $i$.
\end{proof}

Set $C_H := F^\times \varphi(H) \le D^\times$.

\begin{lemma}\label{lem:8.4}
$C_H$ is a nilpotent group, and $Z(C_H) = F^\times$.
\end{lemma}

\begin{proof}
Clearly $F^\times \le Z(C_H)$. Conversely, if $f\varphi(h) \in Z(C_H)$ ($f \in F^\times, h \in H$), then since $f$ is central, $\varphi(h)$ commutes with all of $\varphi(H)$, so $h \in Z(H) \Rightarrow \varphi(h) \in F^\times \Rightarrow f\varphi(h) \in F^\times$.

The natural surjection $\varphi(H) \twoheadrightarrow C_H / F^\times$, $h \mapsto \varphi(h)F^\times$, has kernel $\varphi(H) \cap F^\times$, so $C_H / F^\times \cong \varphi(H) / (\varphi(H) \cap F^\times)$ is a quotient of a nilpotent group and is therefore nilpotent. Writing $\gamma_1(C_H) = C_H$ and $\gamma_{n+1}(C_H) = [\gamma_n(C_H), C_H]$, if $\gamma_{c+1}(C_H / F^\times) = \gamma_{c+1}(C_H)F^\times / F^\times = 1$, then $\gamma_{c+1}(C_H) \le F^\times = Z(C_H)$ and hence $\gamma_{c+2}(C_H) = [\gamma_{c+1}(C_H), C_H] = 1$, so $C_H$ is nilpotent.
\end{proof}

\begin{lemma}\label{lem:8.5}
$\ell|_{C_H} : C_H \to F^\times G$ is injective.
\end{lemma}

\begin{proof}
The kernel $\ker(\ell|_{C_H})$ is a normal subgroup of the nilpotent group $C_H$. If it were nontrivial, Lemmas \ref{lem:8.3} and \ref{lem:8.4} would imply $\ker(\ell|_{C_H}) \cap F^\times \neq 1$; however, $\ell|_{F^\times} = \id_{F^\times}$, so $\ker(\ell|_{C_H}) \cap F^\times = 1$, a contradiction.
\end{proof}

\begin{lemma}\label{lem:8.6}
For $h \in \varphi(H)$,
\[
\nu(h) = 1 \quad \Longleftrightarrow \quad h \in \varphi(Z(H)).
\]
Consequently, $\nu$ induces an injective group homomorphism
\[
\bar{\nu} : H/Z(H) \hookrightarrow G/Z(G).
\]
\end{lemma}

\begin{proof}
If $\nu(h) = 1$, then $f := \ell(h) \in F^\times$, so $\ell(f) = f = \ell(h)$; since $\ell$ is injective on $C_H$ by Lemma \ref{lem:8.5}, we have $h = f \in F^\times$, and hence $h \in \varphi(Z(H))$. Conversely, if $u \in \varphi(Z(H))$, then $u$ commutes with $\varphi(H)$ and hence with all of $D = D_{\Q}(H)$, so $u \in Z(D)^\times = F^\times$; since $\ell|_{F^\times} = \id$, we have $\nu(u) = 1$. Thus $\ker(\nu|_{\varphi(H)}) = \varphi(Z(H))$, and $\bar{\nu}(hZ(H)) := \nu(\varphi(h))$ is a well-defined injective homomorphism $H/Z(H) \hookrightarrow G/Z(G)$.
\end{proof}

\begin{lemma}\label{lem:8.7}
For any $0 \neq a \in \varphi(\Q[H])$, we have $\nu(a) \in \nu(\varphi(H))$.
\end{lemma}

\begin{proof}
Grouping $a$ along $H/Z(H)$-cosets gives $a = \sum_{i=1}^r c_i h_i$, where $h_i \in \varphi(H)$ represent distinct $H/Z(H)$-cosets and $0 \neq c_i \in \varphi(\Q[Z(H)]) \subseteq F$. Each term $c_i h_i$ has lowest coset $\nu(h_i)$, and by Lemma \ref{lem:8.6} these cosets are pairwise distinct; hence the minimum coset of the sum is the smallest $\nu(h_i)$ without cancellation.
\end{proof}

\begin{lemma}\label{lem:8.8}
$\nu(\varphi(H)) = G/Z(G)$.
\end{lemma}

\begin{proof}
Take any $g \in G$. Since $\varphi : D_{\Q}(H) \to D$ is surjective, there exist $0 \neq a, b \in \Q[H]$ such that $g = \varphi(a)\varphi(b)^{-1}$. Then $gZ(G) = \nu(g) = \nu(\varphi(a))\nu(\varphi(b))^{-1} \in \nu(\varphi(H))$ by Lemma \ref{lem:8.7}.
\end{proof}

Thus $\bar{\nu} : H/Z(H) \xrightarrow{\sim} G/Z(G)$ is an isomorphism, and moreover
\[
\ell\bigl(F^\times \varphi(H)\bigr) = F^\times G:
\]
the left-hand side is contained in the right-hand side, and for any $g \in G$, Lemma \ref{lem:8.8} gives $h \in \varphi(H)$ with $\nu(h) = gZ(G)$, so $\ell(h) = fgz$ for some $f \in F^\times$ and $z \in Z(G) \subseteq F^\times$, whence $g = \ell(h)f^{-1}z^{-1} \in F^\times \ell(h)$.

Define $H_1 := \ell(\varphi(H)) \le F^\times G$. Since both $\varphi$ and $\ell|_{C_H}$ are injective, $H_1 \cong H$ and $F^\times H_1 = F^\times G$.

\begin{lemma}\label{lem:8.9}
The natural map $i : \Q[H_1] \longrightarrow D$ is injective.
\end{lemma}

\begin{proof}
By Lemma \ref{lem:8.5}, $\ell : \varphi(H) \xrightarrow{\sim} H_1$ is an isomorphism with $\ell|_{F^\times} = \id$ and $\varphi(Z(H)) \subseteq F^\times$, so $Z(H_1) = \ell(\varphi(Z(H))) = \varphi(Z(H))$, and $\Q[Z(H_1)] \hookrightarrow F$ is injective. If $\Q[H_1] \to D$ were not injective, grouping a nonzero element of the kernel along $H_1/Z(H_1)$-cosets would give
\[
\sum_{i=1}^r c_i h_i = 0, \qquad 0 \neq c_i \in \Q[Z(H_1)] \subseteq F^\times,
\]
where $h_1, \dots, h_r \in H_1$ represent distinct $H_1/Z(H_1)$-cosets. By Lemma \ref{lem:8.6}, the cosets $\nu(h_i) \in G/Z(G)$ are pairwise distinct, so the terms $c_i h_i \in F^\times G$ lie in distinct homogeneous components of $A_G = \bigoplus_{c \in G/Z(G)} Fg_c$, forcing $c_i h_i = 0$ for each $i$, a contradiction since $c_i \in F^\times$ and $h_i \in D^\times$.
\end{proof}

\begin{proposition}\label{prop:8.10}
The division ring of fractions of $\Q[H_1]$ inside $D$ is $D$ itself.
\end{proposition}

\begin{proof}
Since $H_1 \cong H$, $\Q[H_1]$ is an Ore domain, and by Lemma \ref{lem:8.9} the inclusion $\Q[H_1] \hookrightarrow D$ extends to an embedding $D_{\Q}(H_1) \hookrightarrow D$. Since $Z(H_1) = \varphi(Z(H))$, we have $F = \Frac \Q[Z(H_1)] \subseteq D_{\Q}(H_1)$, and since $F^\times H_1 = F^\times G$, we have $G \subseteq F^\times H_1 \subseteq D_{\Q}(H_1)^\times$, whence $\Q[G] \subseteq D_{\Q}(H_1)$ and $D = D_{\Q}(G) \subseteq D_{\Q}(H_1)$. Thus $D = D_{\Q}(H_1)$.
\end{proof}

Henceforth we denote $H_1$ simply by $H$, so that inside $D$ we have
\[
\Q[G], \Q[H] \subset D \text{ faithfully}, \qquad D_{\Q}(G) = D_{\Q}(H) = D, \qquad F^\times G = F^\times H.
\]

\section{The Lowest-Term Coefficient Homomorphism and Central Direct Product Decompositions}

Having realized $H$ as a subgroup of $F^\times G$ satisfying $F^\times G = F^\times H$, we now use the lowest-term rational coefficient homomorphism $c : D^\times \to \Q^\times$ on $\Q[G]$ to normalize out rational scalars, embedding both $G$ and $H$ into $E = (F^\times G) \cap \ker c = KG = KH$. Using the fact that the monomial quotient group of a torsion-free abelian group over a field of characteristic zero is free abelian, we then show that there exist free abelian subgroups $C_G, C_H \le K$ such that $E = G \times C_G = H \times C_H$.

Fix a bi-invariant total ordering $<$ on $G$. For
\[
0 \neq a = \sum_{g \in E_a} q_g g \in \Q[G] \qquad (q_g \in \Q^\times)
\]
with $g_0 = \min E_a$, define the lowest-term coefficient map
\[
c_0(a) := q_{g_0} \in \Q^\times.
\]

\begin{lemma}\label{lem:9.1}
$c_0 : \Q[G] \setminus \{0\} \to \Q^\times$ is multiplicative and thus extends uniquely to a group homomorphism $c : D^\times \to \Q^\times$ satisfying $c(q) = q$ ($q \in \Q^\times$) and $c(g) = 1$ ($g \in G$).
\end{lemma}

\begin{proof}
By the bi-invariance of $<$, if $a, b \in \Q[G] \setminus \{0\}$ have minimum support elements $g_0, h_0$, then $g_0 h_0$ is the unique minimum support element of $ab$ with coefficient $c_0(a)c_0(b)$. Applying Lemma \ref{lem:8.1} yields the unique extension $c : D^\times \to \Q^\times$.
\end{proof}

Define $\eta : H \to D^\times$ by $\eta(h) = c(h)^{-1}h$.

\begin{lemma}\label{lem:9.2}
$\eta$ is an injective group homomorphism; denoting its image by $H^\circ$, we have
\[
H^\circ \cong H, \qquad \Q[H^\circ] = \Q[H] \subset D, \qquad F^\times H^\circ = F^\times H = F^\times G,
\]
and $H^\circ \le \ker c$.
\end{lemma}

\begin{proof}
Since $c|_H$ takes values in the central subgroup $\Q^\times$, $\eta$ is a group homomorphism. If $\eta(h) = 1$, then $h = c(h) \in \Q^\times$; since $\Q[H] \hookrightarrow D$ is faithful, the elements of $H$ are $\Q$-linearly independent in $D$, so $h = 1$ and $\eta$ is injective. Since $h = c(h)\eta(h)$ with $c(h) \in \Q^\times \subseteq F^\times$, $H$ and $H^\circ$ span the same $\Q$-subalgebra $\Q[H^\circ] = \Q[H]$ and satisfy $F^\times H^\circ = F^\times H$. Finally, $c(\eta(h)) = c(h)^{-1}c(h) = 1$, so $H^\circ \le \ker c$.
\end{proof}

Replacing $H$ by $H^\circ$, we may assume henceforth that $G, H \le \ker c$. Define
\[
E := (F^\times G) \cap \ker c, \qquad K := F^\times \cap \ker c.
\]

\begin{lemma}\label{lem:9.3}
We have
\[
F^\times = \Q^\times \times K, \qquad E = KG = KH.
\]
\end{lemma}

\begin{proof}
Every $f \in F^\times$ decomposes uniquely as $f = c(f)\cdot (c(f)^{-1}f) \in \Q^\times \times K$. For any $x = fg \in E$ ($f \in F^\times, g \in G$), $1 = c(x) = c(f)c(g) = c(f)$ gives $f \in K$, so $E = KG$; similarly $F^\times H = F^\times G$ gives $E = KH$.
\end{proof}

\begin{lemma}\label{lem:9.4}
$Z(E) = K$.
\end{lemma}

\begin{proof}
Since $K \le F^\times = Z(D)^\times$, we have $K \le Z(E)$. Conversely, if $x = kg \in Z(E)$ ($k \in K, g \in G$), then for every $g_0 \in G$ we have $1 = [kg, g_0] = [g, g_0]$, so $g \in Z(G) \subseteq F^\times$, whence $x \in F^\times \cap E = K$.
\end{proof}

Set $U := K \cap G$ and $V := K \cap H$.

\begin{lemma}\label{lem:9.5}
We have $U = Z(G)$, $V = Z(H)$, and $F = \Q(U) = \Q(V)$.
\end{lemma}

\begin{proof}
Since $K = Z(E)$ and $G \le E$, we have $U = K \cap G \le Z(G)$; conversely, by Proposition \ref{prop:7.2} and $G \le \ker c$, $Z(G) \le F^\times \cap G \cap \ker c = U$, so $U = Z(G)$ and similarly $V = Z(H)$. Proposition \ref{prop:7.2} then gives $F = \Frac \Q[Z(G)] = \Q(U) = \Q(V)$.
\end{proof}

Since $U = Z(G)$ and $F = \Q(U)$, we have $F^\times / (\Q^\times U) = \Q(U)^\times / (\Q^\times U)$. In our earlier preprint \cite{LinWang2026}, we proved that for any field $k$ of characteristic $0$ and any torsion-free abelian group $U$, the quotient group $k(U)^\times / (k^\times U)$ is a free abelian group (for classical background see also May \cite{May1972,May1979}); to make the present paper self-contained, we include the proof here.

For any field $k$ of characteristic zero and any torsion-free abelian group $A$ written additively, denote formal monomials in $k[A]$ by $X^a$ ($a \in A$), set $k(A) = \Frac k[A]$, and write
\[
\Delta_k(A) := k(A)^\times / (k^\times X^A),
\]
where $k^\times X^A = \{cX^a : c \in k^\times, \, a \in A\} \le k(A)^\times$.

We first treat the single-variable rational-exponent case. Let $F_0$ be a field of characteristic zero and let $t$ be transcendental over $F_0$. In a fixed algebraic closure of $F_0(t)$, choose elements $t^{1/n!}$ compatibly so that $(t^{1/(n+1)!})^{n+1} = t^{1/n!}$ for all $n \ge 1$. For $q = a/n! \in \Q$, set $t^q := (t^{1/n!})^a$, which is independent of the factorial denominator. Write $t^{\Q} := \{t^q : q \in \Q\}$ and $F_0(t^{\Q}) := \bigcup_{n \ge 1} F_0(t^{1/n!})$.

\begin{lemma}\label{lem:9.6-puiseux}
If $\operatorname{char} F_0 = 0$, then the quotient group
\[
\mathcal{P}(F_0, t) := F_0(t^{\Q})^\times / (F_0^\times t^{\Q})
\]
is a free abelian group.
\end{lemma}

\begin{proof}
Set $z_n = t^{1/n!}$ and $P_n = F_0(z_n)^\times / (F_0^\times z_n^{\Z})$. Let $\mathcal{I}_n$ be the set of all monic irreducible polynomials in $F_0[z_n]$ other than $z_n$. Unique factorization in $F_0[z_n]$ gives a group isomorphism
\[
P_n \xrightarrow{\sim} \Z^{(\mathcal{I}_n)}, \qquad [f] \longmapsto (v_r(f))_{r \in \mathcal{I}_n}.
\]
To determine the transition map from level $n$ to $n+1$, set $d = n+1$ and $Y = z_{n+1}$, so $z_n = Y^d$. Fix $p(Z) \in \mathcal{I}_n$; then $p(0) \neq 0$, and since $\operatorname{char} F_0 = 0$, $p$ is separable with $\gcd(p, p') = 1$, so $A(Z)p(Z) + B(Z)p'(Z) = 1$ for some $A, B \in F_0[Z]$. Substituting $Z = Y^d$ gives $A(Y^d)p(Y^d) + B(Y^d)p'(Y^d) = 1$. Since $\frac{d}{dY}p(Y^d) = d Y^{d-1} p'(Y^d)$, $p(0) \neq 0$, and $d \neq 0$ in $F_0$, $p(Y^d)$ is coprime to its derivative and hence square-free in $F_0[Y]$. Moreover, distinct $p_1, p_2 \in \mathcal{I}_n$ yield coprime polynomials $p_1(Y^d)$ and $p_2(Y^d)$, and every $q(Y) \in \mathcal{I}_{n+1}$ divides $p(Y^d)$ for a unique $p \in \mathcal{I}_n$ (namely the minimal polynomial of $(Y + (q))^d$ over $F_0$).

Thus the finite nonempty sets $S_p := \{q \in \mathcal{I}_{n+1} : q(Y) \mid p(Y^d)\}$ ($p \in \mathcal{I}_n$) partition $\mathcal{I}_{n+1}$, and $p(Y^d) = \prod_{q \in S_p} q(Y)$. Under the basis isomorphism above, the transition map sends $p \in \mathcal{I}_n$ to $\sum_{q \in S_p} q$. Choosing one representative $q_p \in S_p$ for each $p$, the set $\left\{\sum_{q \in S_p} q\right\} \cup (S_p \setminus \{q_p\})$ is a basis of $\Z^{(S_p)}$, so $P_n \to P_{n+1}$ is split injective with free abelian cokernel.

Finally, the natural map $\iota_n : P_n \to \mathcal{P}(F_0, t)$ is injective: if $f = a/b \in F_0(z_n)^\times$ ($0 \neq a, b \in F_0[z_n]$) satisfies $f = c t^q$ ($c \in F_0^\times, q \in \Q$), then $a = c t^q b$ in $F_0[\Q]$, so $\supp(a) = q + \supp(b) \subseteq \frac{1}{n!}\Z$, forcing $q \in \frac{1}{n!}\Z$ and $t^q \in z_n^{\Z}$. Writing $P_{n+1} = P_n \oplus C_n$ with $C_n$ free abelian gives $\mathcal{P}(F_0, t) = P_1 \oplus \bigoplus_{n \ge 1} C_n$, which is free abelian.
\end{proof}

\begin{lemma}\label{lem:9.7-overgroup}
If $A \le W$ are torsion-free abelian groups, the natural homomorphism $\Delta_k(A) \to \Delta_k(W)$ induced by inclusion is injective.
\end{lemma}

\begin{proof}
If $f = a/b \in k(A)^\times$ ($0 \neq a, b \in k[A]$) equals $cX^w$ in $k(W)^\times$ ($c \in k^\times, w \in W$), then $a = cX^w b$ in $k[W]$, so $\supp(a) = w + \supp(b)$. Picking any $g \in \supp(b) \subseteq A$ gives $w + g \in \supp(a) \subseteq A$, whence $w \in A$ and $f \in k^\times X^A$.
\end{proof}

\begin{proposition}\label{prop:9.8-divisible}
Let $k$ be a field of characteristic zero and $W$ a $\Q$-vector space. Then $\Delta_k(W)$ is a free abelian group.
\end{proposition}

\begin{proof}
Choose a well-ordered $\Q$-basis $(e_\alpha)_{\alpha < \kappa}$ of $W$, and set $W_\alpha = \bigoplus_{\beta < \alpha} \Q e_\beta$ and $D_\alpha = \Delta_k(W_\alpha)$. By Lemma \ref{lem:9.7-overgroup}, $D_\alpha \hookrightarrow D_\beta$ for $\alpha \le \beta$. At each successor stage $\alpha + 1$, setting $F_\alpha = k(W_\alpha)$ and $t_\alpha = X^{e_\alpha}$, we have $k(W_{\alpha+1}) = F_\alpha(t_\alpha^{\Q})$ and
\[
D_{\alpha+1}/D_\alpha \cong F_\alpha(t_\alpha^{\Q})^\times / (F_\alpha^\times t_\alpha^{\Q}) = \mathcal{P}(F_\alpha, t_\alpha),
\]
which is free abelian by Lemma \ref{lem:9.6-puiseux}; hence $D_{\alpha+1} = D_\alpha \oplus C_\alpha$ with $C_\alpha$ free abelian. At limit ordinals $\lambda$, $D_\lambda = \bigcup_{\alpha < \lambda} D_\alpha$, so transfinite induction yields $\Delta_k(W) = \bigoplus_{\alpha < \kappa} C_\alpha$.
\end{proof}

\begin{theorem}\label{thm:9.9}
Let $k$ be a field of characteristic zero and $A$ an arbitrary torsion-free abelian group. Then the quotient group
\[
\Delta_k(A) = k(A)^\times / (k^\times X^A)
\]
is a free abelian group.
\end{theorem}

\begin{proof}
Embed $A$ into its divisible hull $W = \Q \otimes_{\Z} A$. By Lemma \ref{lem:9.7-overgroup} and Proposition \ref{prop:9.8-divisible}, $\Delta_k(A)$ embeds as a subgroup of the free abelian group $\Delta_k(W)$ and is therefore free abelian.
\end{proof}

Returning to $F = \Q(U) = \Q(V)$, Theorem \ref{thm:9.9} and Lemma \ref{lem:9.3} ($F^\times = \Q^\times \times K$) give
\[
\frac{F^\times}{\Q^\times U} = \frac{\Q^\times K}{\Q^\times U} \cong \frac{K}{U}, \qquad \frac{F^\times}{\Q^\times V} \cong \frac{K}{V},
\]
so $K/U$ and $K/V$ are both free abelian groups.

\begin{lemma}\label{lem:9.6}
There exist free abelian subgroups $C_G, C_H \le K$ such that
\[
K = U \times C_G = V \times C_H \qquad \text{and} \qquad E = G \times C_G = H \times C_H.
\]
In particular, $E/G$ and $E/H$ are free abelian groups.
\end{lemma}

\begin{proof}
Since $K/U$ is free abelian and hence projective as a $\Z$-module, the short exact sequence
\[
0 \longrightarrow U \longrightarrow K \longrightarrow K/U \longrightarrow 0
\]
splits. Choosing a section image $C_G \le K$ gives $K = U \times C_G$, and similarly $K = V \times C_H$. Since $E = KG$, $K = UC_G$, and $U \le G$, we have $E = GC_G$. Moreover, $G \cap C_G \le G \cap K = U$ and $U \cap C_G = 1$, so $G \cap C_G = 1$. Since $C_G \le K = Z(E)$, $E = GC_G = G \times C_G$, and similarly $E = H \times C_H$. Thus $E/G \cong C_G$ and $E/H \cong C_H$ are free abelian.
\end{proof}

\section{Weight-Preserving Isomorphism Between the Quotients \texorpdfstring{$G/L$}{G/L} and \texorpdfstring{$H/L$}{H/L}}

In this section, we set $L = G \cap H$ and study the relationship between the free abelian quotients $P_G = G/L$ and $P_H = H/L$. By quotienting out the isolator $I = \sqrt[E]{E'} \le L$ of the derived subgroup and extending scalars to $\Qbar$, we embed the group algebras of $(G/I) \times K$ and $(H/I) \times K$ into a common commutative auxiliary field $\Omega$, prove that $\Omega = k_{\mathrm{aux}}(P_G) = k_{\mathrm{aux}}(P_H)$ forms a weighted rational function field isomorphism over the weight group $B = E/(KL)$, and apply Theorem \ref{thm:6.1} to obtain a weight-preserving group isomorphism $\beta : G/L \xrightarrow{\sim} H/L$.

Define
\[
L := G \cap H, \qquad P_G := G/L, \qquad P_H := H/L, \qquad B := E/KL.
\]

\begin{lemma}\label{lem:10.1}
We have $G, H, L \triangleleft E$, and
\[
E' = G' = H' \le L.
\]
Moreover, $P_G$ and $P_H$ are free abelian groups, and the natural maps
\[
\rho_G : P_G \twoheadrightarrow B, \qquad \rho_H : P_H \twoheadrightarrow B
\]
are surjective.
\end{lemma}

\begin{proof}
Since $E = KG$ and $K \le Z(E)$, $G \triangleleft E$ and similarly $H \triangleleft E$, so $L = G \cap H \triangleleft E$. Central factors do not contribute to commutators, so $E' = [KG, KG] = G'$ and similarly $E' = H'$, whence $E' = G' = H' \le G \cap H = L$. The natural map $G/L \to E/H$, $gL \mapsto gH$, has kernel $G \cap H = L$ and thus embeds $P_G$ into the free abelian group $E/H$ (Lemma \ref{lem:9.6}), so $P_G$ is free abelian, and similarly for $P_H$. Since $E' \le L \le KL$, $B = E/KL$ is abelian; for any $eKL \in B$, writing $e = kg$ ($k \in K, g \in G$) gives $eKL = gKL = \rho_G(gL)$, so $\rho_G$ is surjective, and similarly $\rho_H$ is surjective.
\end{proof}

Set
\[
I = \sqrt[E]{E'} = \{e \in E : \exists\, n \ge 1, \, e^n \in E'\}, \qquad A := E/I.
\]

\begin{lemma}\label{lem:10.2}
We have $E' \le I \le L$, and $A = E/I$ is a torsion-free abelian group.
\end{lemma}

\begin{proof}
Clearly $E' \le I$, and $A = E/I$ is the quotient of the abelian group $E/E'$ by its torsion subgroup and is therefore torsion-free abelian. If $e \in I$, then $e^n \in E' \le G$ for some $n \ge 1$, so $(eG)^n = G$ in the torsion-free group $E/G$, whence $e \in G$; similarly $e \in H$, so $e \in G \cap H = L$ and $I \le L$.
\end{proof}

We write $A = E/I$ additively and denote the image of $e \in E$ in $A$ by $\bar{e} := eI \in A$. Since $I \le L \le KL$, the natural surjection $E \twoheadrightarrow B$ factors through $A$:
\[
\begin{aligned}
E &\longrightarrow A \longrightarrow B, \\
e &\longmapsto eI \longmapsto eKL.
\end{aligned}
\]

Set
\[
\widehat{F} := \Frac(\Qbar \otimes_{\Q} F).
\]

\begin{lemma}\label{lem:10.3}
$\Qbar \otimes_{\Q} F$ is a field (in particular, a domain), and $\widehat{F} = \Qbar(U) = \Qbar(V)$, where $\Qbar(U) = \Frac \Qbar[U]$. Moreover, distinct elements of $U$ (respectively $V$) remain $\Qbar$-linearly independent in $\widehat{F}$.
\end{lemma}

\begin{proof}
For any finitely generated subgroup $U_0 \le U$, $U_0 \cong \Z^r$ is free abelian, so $\Q[U_0] \cong \Q[t_1^{\pm 1}, \dots, t_r^{\pm 1}]$ and $F_0 := \Frac \Q[U_0] \cong \Q(t_1, \dots, t_r)$ is purely transcendental over $\Q$. For every finite extension $L/\Q$, $L \otimes_{\Q} F_0 \cong L(t_1, \dots, t_r)$ is a field; taking the directed union over all finite subextensions $L \subseteq \Qbar$ shows that $\Qbar \otimes_{\Q} \Q(U_0) \cong \Qbar(U_0)$ is a field. Since $F = \Q(U) = \bigcup_{U_0} \Q(U_0)$ is the directed union over finitely generated subgroups $U_0 \le U$ and tensor products commute with directed colimits while $\Qbar$ is flat over $\Q$, we obtain
\[
\Qbar \otimes_{\Q} F = \bigcup_{U_0} \Qbar(U_0) = \Qbar(U),
\]
which is a field, and similarly $\widehat{F} = \Qbar(V)$. Given distinct $u_1, \dots, u_s \in U$, choosing a finitely generated subgroup $U_0 \le U$ containing them shows that $u_1, \dots, u_s$ are part of the $\Qbar$-basis of $\Qbar[U_0] \subseteq \Qbar(U_0) \subseteq \widehat{F}$, and hence are $\Qbar$-linearly independent in $\widehat{F}$ (and similarly for $V$).
\end{proof}

Set $\Omega := \Frac \widehat{F}[A]$. For $(a, k) \in A \times K$ (with group operation $(a, k)(a', k') = (a + a', kk')$), define
\[
\Theta(a, k) := k X^{a - \bar{k}} \in \Omega^\times.
\]

\begin{lemma}\label{lem:10.4}
$\Theta : A \times K \longrightarrow \Omega^\times$ is an injective group homomorphism.
\end{lemma}

\begin{proof}
Since $\overline{kk'} = \bar{k} + \bar{k}'$, we have $\Theta(a, k)\Theta(a', k') = kk' X^{a + a' - \bar{k} - \bar{k}'} = \Theta(a + a', kk')$. If $\Theta(a, k) = k X^{a - \bar{k}} = 1$, linear independence of distinct monomials in $\widehat{F}[A]$ over $\widehat{F}$ gives $a - \bar{k} = 0$ and $k = 1$, whence $a = 0$.
\end{proof}

Define the three abelian subgroups of $A \times K$:
\[
\mathcal{A}_G = (G/I) \times K, \qquad \mathcal{A}_H = (H/I) \times K, \qquad \mathcal{A}_L = (L/I) \times K.
\]

\begin{lemma}\label{lem:10.5}
$\Theta$ induces an injective $\Qbar$-algebra homomorphism
\[
\Qbar[\mathcal{A}_G] \hookrightarrow \widehat{F}[A],
\]
as well as the corresponding embeddings for $\mathcal{A}_H$ and $\mathcal{A}_L$.
\end{lemma}

\begin{proof}
We prove the statement for $G$. Suppose there is a relation
\[
\sum_{j=1}^s c_j \Theta(a_j, k_j) = 0, \qquad c_j \in \Qbar,
\]
with $(a_j, k_j) \in \mathcal{A}_G$ pairwise distinct. Setting $e_j = a_j - \bar{k}_j \in A$, we have
\[
\sum_{j=1}^s c_j k_j X^{e_j} = 0.
\]
Since distinct monomials $X^e$ in $\widehat{F}[A]$ are linearly independent over $\widehat{F}$, for each exponent $e$ that appears we must have
\[
\sum_{j : e_j = e} c_j k_j = 0 \quad \text{in } \widehat{F}.
\]
Fix one index $j_0$ with $e_{j_0} = e$ and write $(a_0, k_0) = (a_{j_0}, k_{j_0})$. For every $j$ with $e_j = e$, we have
\[
a_j - \bar{k}_j = a_0 - \bar{k}_0 \implies \overline{k_j k_0^{-1}} = a_j - a_0 \in G/I.
\]
Since $I \le G$ and $k_j k_0^{-1} \in K$, this implies $u_j := k_j k_0^{-1} \in K \cap G = U$. Moreover, these elements $u_j$ are pairwise distinct: if $k_j k_0^{-1} = k_{j'} k_0^{-1}$, then $k_j = k_{j'}$, and $e_j = e_{j'}$ then gives $a_j = a_{j'}$, contradicting the assumption that the pairs $(a_j, k_j)$ are distinct. Multiplying $\sum_{j : e_j = e} c_j k_j = 0$ by $k_0^{-1}$ yields
\[
\sum_{j : e_j = e} c_j u_j = 0,
\]
and by Lemma \ref{lem:10.3}, $c_j = 0$ for all $j$ with $e_j = e$. Since this holds for every $e$, all $c_j = 0$. Replacing $U$ by $V$ proves the claim for $\mathcal{A}_H$, and $\mathcal{A}_L \le \mathcal{A}_G$ follows immediately.
\end{proof}

\begin{lemma}\label{lem:10.6}
We have
\[
\Qbar(\Theta(\mathcal{A}_G)) = \Omega = \Qbar(\Theta(\mathcal{A}_H)).
\]
\end{lemma}

\begin{proof}
For $u \in U$, $\Theta(\bar{u}, u) = u$, so $\Qbar(\Theta(\mathcal{A}_G))$ contains $\Qbar(U) = \widehat{F}$. For $a \in G/I$, $\Theta(a, 1) = X^a$, and for $k \in K$, $\Theta(0, k) = k X^{-\bar{k}}$ gives $X^{\bar{k}} = k\Theta(0, k)^{-1}$. Since $E = KG$, $A = E/I$ is generated by $G/I$ and $\bar{K}$, so $\Qbar(\Theta(\mathcal{A}_G))$ contains $X^a$ for all $a \in A$ and thus equals $\Omega$. The proof for $H$ is identical using $\widehat{F} = \Qbar(V)$ and $E = KH$.
\end{proof}

Set $k_{\mathrm{aux}} := \Qbar(\Theta(\mathcal{A}_L)) \subseteq \Omega$; clearly $\Qbar \subseteq k_{\mathrm{aux}}$.

\begin{lemma}\label{lem:10.7}
The natural isomorphisms
\[
\frac{\mathcal{A}_G}{\mathcal{A}_L} \cong \frac{G}{L} = P_G, \qquad \frac{\mathcal{A}_H}{\mathcal{A}_L} \cong \frac{H}{L} = P_H
\]
yield $\Omega = k_{\mathrm{aux}}(P_G) = k_{\mathrm{aux}}(P_H)$, where both sides are purely transcendental extensions of $k_{\mathrm{aux}}$ generated by bases of $P_G$ and $P_H$.
\end{lemma}

\begin{proof}
Clearly $\mathcal{A}_G/\mathcal{A}_L = ((G/I) \times K)/((L/I) \times K) \cong (G/I)/(L/I) \cong G/L = P_G$. Since $P_G$ is free abelian by Lemma \ref{lem:10.1}, the short exact sequence
\[
0 \longrightarrow \mathcal{A}_L \longrightarrow \mathcal{A}_G \longrightarrow P_G \longrightarrow 0
\]
splits; choose a section $R_G \le \mathcal{A}_G$ so that $\mathcal{A}_G = \mathcal{A}_L \oplus R_G$ with $R_G \cong P_G$.

For any distinct $r_1, \dots, r_s \in R_G$, if $\sum_{i=1}^s f_i \Theta(r_i) = 0$ with $f_i \in k_{\mathrm{aux}}$, clearing common denominators allows us to assume $f_i \in \Qbar[\Theta(\mathcal{A}_L)]$. Since $\mathcal{A}_L \oplus R_G$ is a direct sum, distinct $r_i$ lie in distinct $\mathcal{A}_L$-cosets, so expanding each $f_i$ produces a $\Qbar$-linear combination of distinct elements of $\mathcal{A}_G$; Lemma \ref{lem:10.5} then forces $f_i = 0$ for all $i$. Thus $\Theta(R_G)$ is linearly independent over $k_{\mathrm{aux}}$. Choosing a basis $(r_\lambda)_{\lambda \in \Lambda}$ of $R_G$, any polynomial relation $\sum_e f_e \prod_\lambda \Theta(r_\lambda)^{e_\lambda} = \sum_e f_e \Theta\left(\sum_\lambda e_\lambda r_\lambda\right) = 0$ over $k_{\mathrm{aux}}$ is a $k_{\mathrm{aux}}$-linear relation among distinct elements of $\Theta(R_G)$, whence $f_e = 0$. Therefore $(\Theta(r_\lambda))_{\lambda \in \Lambda}$ are algebraically independent over $k_{\mathrm{aux}}$, and Lemma \ref{lem:10.6} gives $\Omega = \Qbar(\Theta(\mathcal{A}_G)) = k_{\mathrm{aux}}(P_G)$, and similarly $\Omega = k_{\mathrm{aux}}(P_H)$.
\end{proof}

With $B = E/KL$ as the weight group, for each character $\chi : B \to \Qbar^\times$ we define on $\widehat{F}[A]$:
\[
\sigma_\chi(f) = f \quad (f \in \widehat{F}), \qquad \sigma_\chi(X^{eI}) = \chi(eKL)X^{eI} \quad (e \in E).
\]

\begin{lemma}\label{lem:10.8}
$\sigma_\chi$ is well-defined and extends to a field automorphism of $\Omega$, under which the weight of $\Theta(a, k)$ for $(a, k) \in A \times K$ is the image of $a$ in $B$. In particular:
\begin{enumerate}
\item[(i)] $\sigma_\chi$ fixes $k_{\mathrm{aux}}$ pointwise;
\item[(ii)] on $P_G$, the induced weight map is precisely $\rho_G : P_G \twoheadrightarrow B$;
\item[(iii)] on $P_H$, the induced weight map is precisely $\rho_H : P_H \twoheadrightarrow B$.
\end{enumerate}
\end{lemma}

\begin{proof}
If $eI = e'I$, then $e^{-1}e' \in I \le L \le KL$, so $eKL = e'KL$; thus $\sigma_\chi$ is a well-defined ring automorphism of $\widehat{F}[A]$ with inverse $\sigma_{\chi^{-1}}$ and extends to $\Omega$. For $a = eI \in A$ and $k \in K \subseteq \widehat{F}^\times$, we have $\Theta(a, k) = k X^{(ek^{-1})I}$; since $k \in \widehat{F}$ is fixed by $\sigma_\chi$ and $ek^{-1}KL = eKL$, we obtain $\sigma_\chi(\Theta(a, k)) = \chi(eKL)\Theta(a, k)$.

If $(a, k) \in \mathcal{A}_L$, then $a = lI$ for some $l \in L$, so $\chi(lKL) = 1$ and $\sigma_\chi(\Theta(a, k)) = \Theta(a, k)$; since $\sigma_\chi$ also fixes $\Qbar$, it fixes $k_{\mathrm{aux}}$ pointwise. Under the splitting $\mathcal{A}_G = \mathcal{A}_L \oplus R_G$ with isomorphism $q : R_G \xrightarrow{\sim} P_G$, $(gI, k) \mapsto gL$, each $p = gL \in P_G$ with $q^{-1}(p) = (gI, k) \in R_G$ corresponds to the monomial $t_p := \Theta(gI, k) = k X^{(gk^{-1})I}$, which satisfies
\[
\sigma_\chi(t_p) = \chi(gk^{-1}KL) t_p = \chi(gKL) t_p = \chi(\rho_G(p)) t_p,
\]
and similarly for $P_H$.
\end{proof}

Taking $\Phi = \id_\Omega$ on $\Omega = k_{\mathrm{aux}}(P_G) = k_{\mathrm{aux}}(P_H)$ therefore gives a $B$-weighted $k_{\mathrm{aux}}$-field isomorphism over $k_{\mathrm{aux}} \supseteq \Qbar$. Applying Theorem \ref{thm:6.1} yields:

\begin{corollary}\label{cor:10.9}
There exists a group isomorphism
\[
\beta : P_G \xrightarrow{\sim} P_H
\]
such that
\[
\rho_H \beta = \rho_G.
\]
\end{corollary}

\section{Proof of Theorem \ref{thm:main}}

In this section, we combine the weight-preserving quotient isomorphism $\beta : G/L \xrightarrow{\sim} H/L$ from Section 10 to complete the proof of Theorem \ref{thm:main}: we lift the difference map $\beta(p) - p \in KL/L$ along the free abelian group $G/L$ to a central homomorphism $\eta : G/L \to K$, and prove that the centrally modified map $\Psi(g) = g\eta(gL)$ is a group isomorphism $G \xrightarrow{\sim} H$.

Since $E' \le L$, $E/L$ is an abelian group (written additively here), containing $P_G = G/L$ and $P_H = H/L$ as subgroups. By Corollary \ref{cor:10.9}, for every $p \in P_G$ we have
\[
\beta(p) - p \in \ker(E/L \longrightarrow E/KL) = KL/L.
\]
The second isomorphism theorem gives a natural surjection $K \twoheadrightarrow KL/L$, $k \mapsto kL$.

\begin{lemma}\label{lem:11.1}
There exists a group homomorphism $\eta : P_G \to K$ such that for every $p \in P_G$,
\[
\eta(p)L = \beta(p) - p \quad \text{in } E/L.
\]
\end{lemma}

\begin{proof}
Since $E/L$ is abelian and $\beta$ is a homomorphism, $d(p) := \beta(p) - p$ defines a group homomorphism $d : P_G \to KL/L$. Because $P_G$ is free abelian and hence projective as a $\Z$-module, $d$ lifts along the surjection $K \twoheadrightarrow KL/L$ to a homomorphism $\eta : P_G \to K$.
\end{proof}

Returning to multiplicative notation in $E$, define
\[
\begin{aligned}
\Psi : G &\longrightarrow E, \\
g &\longmapsto g\eta(gL),
\end{aligned}
\qquad \eta(gL) \in K = Z(E).
\]

\begin{lemma}\label{lem:11.2}
$\Psi$ is a group homomorphism, $\Psi(G) \le H$, and $\Psi|_L = \id_L$. Moreover, $\Psi(g)L = \beta(gL)$ in $E/L$.
\end{lemma}

\begin{proof}
Since $\eta(gL) \in K = Z(E)$ is central and $g \mapsto \eta(gL)$ is a homomorphism,
\[
\Psi(g_1 g_2) = g_1 g_2 \eta(g_1 g_2 L) = g_1 g_2 \eta(g_1 L)\eta(g_2 L) = \Psi(g_1)\Psi(g_2).
\]
By Lemma \ref{lem:11.1}, in $E/L$ (written additively) we have
\[
\Psi(g)L = gL + \eta(gL)L = gL + \beta(gL) - gL = \beta(gL) \in H/L.
\]
Since $L \le H$, the full preimage of $H/L$ under $E \twoheadrightarrow E/L$ is $H$, so $\Psi(g) \in H$. For $l \in L$, $lL = L$ and $\eta(L) = 1$, so $\Psi(l) = l$.
\end{proof}

\begin{theorem}\label{lem:11.3}
$\Psi : G \longrightarrow H$ is a group isomorphism.
\end{theorem}

\begin{proof}
We first prove injectivity: if $\Psi(g) = 1$, then $\beta(gL) = \Psi(g)L = L$; since $\beta$ is injective, $gL = L$, so $g \in L$, and since $\Psi|_L = \id_L$, we obtain $g = \Psi(g) = 1$.

Next we prove surjectivity: for any $h \in H$, since $\beta : P_G \to P_H$ is surjective, there exists $g \in G$ such that $\beta(gL) = hL$. Then $\Psi(g)L = \beta(gL) = hL$, so $l := \Psi(g)^{-1}h \in L$, and hence
\[
\Psi(gl) = \Psi(g)\Psi(l) = \Psi(g)l = h.
\]
Thus $\Psi : G \xrightarrow{\sim} H$ is a group isomorphism.
\end{proof}

This completes the proof of $G \cong H$.

\section{Extensions to Base Fields of Characteristic Zero and Torsion-Free Hypercentral Groups}

In this section, we show that the proofs in Sections 7--11 apply not only over any base field $k$ of characteristic zero, but also carry over without substantive change to arbitrary torsion-free hypercentral groups (as well as general torsion-free groups satisfying bi-orderability, the Ore condition, and nontrivial central intersection of normal subgroups).

First, because the preceding arguments used no special arithmetic properties of $\Q$ beyond its being the base field of characteristic zero and admitting an algebraic closure in Section 10, we immediately obtain:

\begin{corollary}\label{cor:12.1}
Let $k$ be any field of characteristic zero, and let $G$ and $H$ be torsion-free nilpotent groups. If there exists a $k$-algebra isomorphism
\[
D_k(G) \cong_k D_k(H),
\]
then $G \cong H$.
\end{corollary}

\begin{proof}
Replace $\Q$ by $k$ and $\Qbar$ by an algebraic closure $\bar{k}$ throughout; the bi-ordering and Ore division ring constructions \cite{Malcev1951,Neumann1949,Tamari1954,BartholdiKielak2019} hold over any field, and Theorem \ref{thm:9.9} (\cite{LinWang2026}) holds over any field of characteristic zero.
\end{proof}

Inspecting the arguments in Sections 7--11 shows that the proof does not depend on the strong hypothesis that the lower central series terminates in finitely many steps. Specifically, recall that the transfinite upper central series $(Z_\alpha(K))_{\alpha \ge 0}$ of a group $K$ is defined by
\[
Z_0(K) = 1, \qquad Z_{\alpha+1}(K)/Z_\alpha(K) = Z\bigl(K/Z_\alpha(K)\bigr),
\]
and $Z_\lambda(K) = \bigcup_{\alpha < \lambda} Z_\alpha(K)$ for limit ordinals $\lambda$. A group $K$ is called a \textbf{hypercentral group} (or $ZA$-group) if $Z_\Lambda(K) = K$ for some ordinal $\Lambda$. We now verify that the three places in Sections 7 and 8 where nilpotency was invoked hold in full generality for arbitrary torsion-free hypercentral groups.

\begin{lemma}\label{lem:12.2}
If $K$ is a hypercentral group and $N \triangleleft K$ with $N \neq 1$, then $N \cap Z(K) \neq 1$.
\end{lemma}

\begin{proof}
Let $Z_\Lambda(K) = K$. Since $N \cap Z_\Lambda(K) = N \neq 1$, by well-ordering we may choose the least ordinal $\alpha \le \Lambda$ such that $N \cap Z_\alpha(K) \neq 1$. Clearly $\alpha \neq 0$, and $\alpha$ cannot be a limit ordinal because $N \cap Z_\alpha(K) = \bigcup_{\beta < \alpha} (N \cap Z_\beta(K)) \neq 1$ would imply $N \cap Z_\beta(K) \neq 1$ for some $\beta < \alpha$, contradicting the minimality of $\alpha$. Thus $\alpha = \beta + 1$ is a successor ordinal and $N \cap Z_\beta(K) = 1$. Take any $1 \neq x \in N \cap Z_{\beta+1}(K)$. For every $g \in K$, since $Z_{\beta+1}(K)/Z_\beta(K) = Z(K/Z_\beta(K))$ and $N \triangleleft K$, we have $[x, g] \in N \cap Z_\beta(K) = 1$; hence $1 \neq x \in N \cap Z(K)$.
\end{proof}

\begin{lemma}\label{lem:12.3}
If a central subgroup $Z \le Z(C)$ of a group $C$ has the property that $C/Z$ is hypercentral, then $C$ is also hypercentral. In particular, if $H$ is hypercentral, then the subgroup $C_H = F^\times \varphi(H)$ in Section 8 is hypercentral with $Z(C_H) = F^\times$.
\end{lemma}

\begin{proof}
Let $\pi : C \twoheadrightarrow C/Z$ be the natural quotient map. Transfinite induction on $\alpha$ shows that $\pi^{-1}(Z_\alpha(C/Z)) \subseteq Z_{1+\alpha}(C)$: for $\alpha = 0$, $\pi^{-1}(1) = Z \le Z_1(C)$; if the inclusion holds for $\alpha$ and $x \in \pi^{-1}(Z_{\alpha+1}(C/Z))$, then for every $y \in C$ we have $\pi([x, y]) = [\pi(x), \pi(y)] \in Z_\alpha(C/Z)$, so $[x, y] \in \pi^{-1}(Z_\alpha(C/Z)) \subseteq Z_{1+\alpha}(C)$ and hence $x \in Z_{1+(\alpha+1)}(C)$; limit ordinals follow by taking unions. Thus $Z_\Lambda(C/Z) = C/Z$ implies $Z_{1+\Lambda}(C) = C$. For $C_H = F^\times \varphi(H)$, the proof of $Z(C_H) = F^\times$ in Lemma \ref{lem:8.4} used only that $\varphi$ is a division ring isomorphism, while $C_H/F^\times \cong \varphi(H)/(\varphi(H)\cap F^\times)$ is a quotient of the hypercentral group $H$ and is therefore hypercentral; hence $C_H$ is hypercentral.
\end{proof}

\begin{lemma}\label{lem:12.4}
If $G$ is a torsion-free hypercentral group, then $G$ is bi-orderable and amenable, $Z(G)$ is isolated in $G$, and the quotient group $G/Z(G)$ is again a torsion-free hypercentral group (and hence also bi-orderable).
\end{lemma}

\begin{proof}
By a standard theorem in group theory (see Robinson \cite[12.2.3]{Robinson1996}), every hypercentral group is locally nilpotent; that is, every finitely generated subgroup is nilpotent. Consequently, every finitely generated subgroup of a torsion-free hypercentral group $G$ is torsion-free nilpotent, and hence bi-orderable and amenable by Malcev's theorem \cite{Malcev1951}. Because both bi-orderability and amenability are local properties (closed under directed unions), $G$ itself is bi-orderable and amenable. Unique extraction of roots in bi-orderable groups (Lemma \ref{lem:7.3}, exactly as in Lemma \ref{lem:7.4}) implies that $Z(G)$ is isolated in $G$, so $G/Z(G)$ is torsion-free. Moreover, $Z_\alpha(G/Z(G)) = Z_{1+\alpha}(G)/Z(G)$ for every ordinal $\alpha$, so $G/Z(G)$ is again a torsion-free hypercentral group and hence bi-orderable.
\end{proof}

\begin{theorem}\label{thm:12.5}
Let $k$ be any field of characteristic zero, and let $G$ and $H$ be arbitrary torsion-free hypercentral groups. If there exists a $k$-algebra isomorphism
\[
D_k(G) \cong_k D_k(H)
\]
(in particular, when $k = \Q$, if there exists a ring isomorphism $D_{\Q}(G) \cong D_{\Q}(H)$), then $G \cong H$.

More generally, let $G$ and $H$ be torsion-free groups satisfying:
\begin{enumerate}
\item[(i)] $G$ and $H$ admit bi-invariant total orderings, and the group algebras $k[G]$ and $k[H]$ satisfy the classical Ore condition (for instance, when $G$ and $H$ are amenable);
\item[(ii)] the central quotients $G/Z(G)$ and $H/Z(H)$ admit bi-invariant total orderings;
\item[(iii)] under the isomorphism $\varphi : D_k(H) \xrightarrow{\sim} D_k(G)$, every nontrivial normal subgroup of $C_H = Z(D_k(G))^\times \varphi(H)$ intersects the center $Z(C_H) = Z(D_k(G))^\times$ nontrivially.
\end{enumerate}
Then $D_k(G) \cong_k D_k(H)$ implies $G \cong H$.
\end{theorem}

\begin{proof}
For torsion-free groups $G$ and $H$ satisfying conditions (i)--(iii): condition (i) guarantees the center characterization $Z(D_k(G)) = \Frac k[Z(G)]$ of Proposition \ref{prop:7.2} and the lowest-term coefficient homomorphism $c$ of Section 9; condition (ii) guarantees that on the centrally localized strongly graded ring $A_G$ in Section 8 one can define the multiplicative leading-term map $\ell_0$ and extend it to $\ell : D^\times \to F^\times G$; condition (iii) replaces Lemmas \ref{lem:8.3}--\ref{lem:8.4} to give the injectivity of $\ell|_{C_H}$ in Lemma \ref{lem:8.5}, thereby completing Lemmas \ref{lem:8.6}--\ref{lem:8.9} and Proposition \ref{prop:8.10} to replace $H$ by the isomorphic subgroup $H_1 \le F^\times G$. The arguments in Sections 9--11 depend only on the central decomposition $E = KG = KH$ with $K = Z(E)$ and use no further nilpotency assumptions, yielding $G \cong H$. When $G$ and $H$ are torsion-free hypercentral groups, Lemmas \ref{lem:12.2}, \ref{lem:12.3}, and \ref{lem:12.4} ensure that conditions (i)--(iii) hold automatically.
\end{proof}

\begin{remark}\label{rem:12.6}
In their 1982 paper \cite{Farkas1982}, Farkas, Schofield, Snider, and Stafford constructed two non-isomorphic finitely generated torsion-free supersolvable groups (that is, polycyclic groups possessing a finite normal series with cyclic factors) whose rational group algebras have isomorphic classical Ore division rings of fractions. This counterexample demonstrates that once one moves beyond torsion-free nilpotent and torsion-free hypercentral groups to torsion-free supersolvable groups—where the group center may be trivial and nontrivial normal subgroups need not intersect the center (so that condition (iii) of Theorem \ref{thm:12.5} fails)—isomorphism of the division rings of fractions no longer suffices to determine the isomorphism class of the underlying group. Thus, within the classical hierarchy of solvable groups, Theorems \ref{thm:main} and \ref{thm:12.5} reach the optimal class of groups for which this reconstruction property holds.
\end{remark}

\section{Concluding Remarks and Future Directions}

We have proved that torsion-free nilpotent groups and torsion-free hypercentral groups of arbitrary cardinality are uniquely determined up to isomorphism by the classical Ore division rings of fractions of their rational (or characteristic-zero) group algebras, resting on a weight-preserving isomorphism theorem for rational function fields of free abelian groups of arbitrary rank. Building on the methods and results developed in this paper, several directions merit further investigation in future work:

\begin{enumerate}
\item \textbf{Reconstruction of general (not necessarily torsion-free) abelian groups via fraction fields of quotients of group algebras modulo prime ideals:} In the commutative setting, we already know that any torsion-free abelian group $G$ is uniquely determined by the fraction field $\Q(G) = \Frac \Q[G]$ of its rational group algebra \cite{LinWang2026}. For a general abelian group $G$ that is not necessarily torsion-free (with torsion subgroup $T_G = \Tor(G)$), the group algebra $k[G]$ contains zero divisors whenever $T_G \neq 0$ and thus does not admit a classical total field of fractions directly; a natural idea is to ask whether $G$ can be reconstructed by quotienting the group algebra by suitable prime ideals $\mathfrak{p}$ and passing to the fraction fields $\Frac(k[G]/\mathfrak{p})$ of the resulting quotient domains. It should be noted that without restrictions on the prime ideals and the base field, the quotient by a single prime ideal cannot recover the group: for instance, every group algebra admits the augmentation homomorphism $\varepsilon_G : k[G] \to k$ ($\sum a_g X^g \mapsto \sum a_g$), whose kernel $I_G = \ker \varepsilon_G$ is a maximal (and hence prime) ideal satisfying $k[G]/I_G \cong k$ for every group $G$; moreover, over an algebraically closed field $k$ of characteristic zero, the non-isomorphic groups $G = \Z \oplus C_4$ and $H = \Z \oplus (C_2 \oplus C_2)$ have isomorphic group algebras $k[G] \cong k[H]$, so the prime spectrum and its quotient rings over an algebraically closed field cannot recover the full torsion structure.

However, if one focuses on \textbf{minimal prime ideals} and fixes the base field to be the rational numbers $\Q$ (over which finite abelian groups are uniquely determined by their rational group algebras by the classical theorem of Perlis and Walker \cite{PerlisWalker1950}; for example, $\Q[C_4] \cong \Q \times \Q \times \Q(i)$ is not isomorphic to $\Q[C_2^2] \cong \Q^4$), a structure emerges that connects directly with our methods: over an algebraically closed field $k$ of characteristic zero, the divisibility of $k^\times$ allows every character $\chi : T_G \to k^\times$ on the torsion subgroup to extend to a character $\tilde{\chi} : G \to k^\times$, inducing a surjective ring homomorphism
\[
\varphi_\chi : k[G] \twoheadrightarrow k[G/T_G], \qquad X^g \longmapsto \tilde{\chi}(g)X^{g + T_G},
\]
whose kernel $\mathfrak{p}_\chi = \ker \varphi_\chi$ is a minimal prime ideal of $k[G]$ with quotient fraction field $\Frac(k[G]/\mathfrak{p}_\chi) \cong k(G/T_G)$. Accordingly, a compelling question for future study is whether an arbitrary abelian group $G$ can be reconstructed from the family of fraction fields $\{\Frac(\Q[G]/\mathfrak{p})\}$ of the quotients of $\Q[G]$ modulo its minimal prime ideals $\mathfrak{p}$, together with the rational arithmetic structure of these branches and their Galois action correspondence upon base change to $\Qbar$. This approach naturally divides into two steps: first reconstructing the torsion-free quotient $G/T_G$ from the fraction fields of the branches, and then using the rational structure and Galois action across the branches to recover the torsion subgroup $T_G$ as well as the group extension structure
\[
0 \longrightarrow T_G \longrightarrow G \longrightarrow G/T_G \longrightarrow 0.
\]
\item \textbf{Residually torsion-free nilpotent groups and universal division rings of fractions:} For free groups, surface groups, or more general residually torsion-free nilpotent groups, $k[G]$ no longer satisfies the classical Ore condition when $G$ is non-amenable, but by the work of Linnell, Kielak, and Jaikin-Zapirain, its natural division ring envelope is the Hughes-free division ring (or Linnell universal division ring of fractions) $\mathcal{D}_{k[G]}$. Investigating the extent to which an isomorphism $\mathcal{D}_{\Q[G]} \cong \mathcal{D}_{\Q[H]}$ determines the isomorphism class or finer group-theoretic invariants of residually torsion-free nilpotent groups is a natural next step.
\item \textbf{Base fields of positive characteristic:} Our proof uses the characteristic-zero hypothesis in two essential places: first, in Theorem \ref{thm:9.9}, which relies on the separability and square-freeness of $p(Y^d)$ in characteristic zero; and second, in the character separation lemma via the existence of primitive roots of unity of arbitrary order and the torsion-free divisibility of the multiplicative group of an algebraically closed subfield. Over a base field of characteristic $p > 0$, where the $p$-power Frobenius endomorphism and purely inseparable extensions enter, formulating an analogue of weighted rational function field rigidity up to perfection (or modulo $p$-power roots) is an interesting problem in its own right.
\item \textbf{Infinite-dimensional algebraic tori and equivariant birational geometry:} From an algebraic-geometric perspective, Theorem \ref{thm:6.1} can be viewed as a rigidity theorem for infinite-dimensional split algebraic tori under birational transformations equivariant with respect to a diagonalizable group action. Extending our modulo-$p$ defect conservation, countable zero-weight summand absorption, and transfinite block upper-triangular patching techniques to non-split algebraic tori or broader classes of infinite-dimensional equivariant varieties promises further applications in infinite-dimensional birational geometry.
\end{enumerate}

\section*{Author Contributions and AI Usage Statement}

\noindent\textbf{Human Contributions:} The human author verified the ideas and proofs generated by AI, and organized and rewrote the proofs; in addition, the human author independently proposed Proposition \ref{prop:7.2} and Proposition \ref{prop:8.10} and independently provided their proofs.

\vspace{4pt}
\noindent\textbf{AI Participation:} AI provided the initial ideas and preliminary proofs for the paper, and participated in code writing.

\end{document}